\documentclass[11p,reqno]{amsart}
\usepackage{amsmath}
\usepackage{amsfonts}
\usepackage{amssymb}
\usepackage{graphicx}
\usepackage{color}
\newtheorem{theorem}{Theorem}[section]
\newtheorem*{theorem*}{Theorem}
\newtheorem{lemma}{Lemma}[section]
\newtheorem{corollary}[theorem]{Corollary}
\newtheorem{proposition}{Proposition}[section]

\newtheorem{conjecture}[theorem]{Conjecture}

\newtheorem{remark}[theorem]{Remark}

\def\Ric{\text{Ric}}

\def\p{\partial}

\def\aint{\frac{\ \ }{\ \ }{\hskip -0.4cm}\int}

\def\Ric{\operatorname{Ric}}

\numberwithin{equation}{section}

\begin{document}
	\title[Comparison and Vanishing Theorems]{Comparison and vanishing  theorems for K\"ahler manifolds}

\author{Lei Ni}\thanks{The research of LN is partially supported by NSF grant DMS-1401500.  }
\address{Lei Ni. Department of Mathematics, University of California, San Diego, La Jolla, CA 92093, USA}
\email{lni@math.ucsd.edu}

\author{Fangyang Zheng} \thanks{The research of FZ is partially supported by a Simons Collaboration Grant 355557.}
\address{Fangyang Zheng. Department of Mathematics,
The Ohio State University, Columbus, OH 43210,
USA}
\email{{zheng.31@osu.edu}}


\maketitle
\begin{abstract} In this paper, we consider orthogonal Ricci curvature $Ric^{\perp}$ for K\"ahler manifolds, which is a curvature condition closely related to Ricci curvature and holomorphic sectional curvature. We prove comparison theorems and a vanishing theorem related to these curvature conditions, and construct various examples to illustrate their subtle relationship. As a consequence of the vanishing theorem, we show that any compact K\"ahler manifold with positive orthogonal Ricci curvature must be projective. The simply-connectedness is also shown when the complex dimension is smaller than five.
\end{abstract}

\section{Introduction}

\medskip

There are several relatively recent works on comparison theorems on K\"ahler manifolds. In \cite{LW}, for a K\"ahler manifold $(M^m, g)$, Li and Wang introduced the condition ``bisectional curvature bounded from below by a constant $\lambda$" defined as
\begin{equation}\label{eq:lw-1}
R(Z, \overline{Z}, W, \overline{W})\ge \lambda \left( |Z|^2 |W|^2+\left|\langle Z, \overline{W}\rangle\right|^2\right)
\end{equation}
for any $(1, 0)$ vectors $Z, W\in T'M$ satisfying either $\langle Z, \overline{W}\rangle =0$ or $Z=W$, where the complexified tangent space $T_\mathbb{C} M=T'M \oplus T''M$, namely is  decomposed (with respect to the almost complex structure $C$) into the holomorphic subspace ($T'M$) and antiholomorphic subspace ($T''M$). Let $m$ be the complex dimension of $M$ and $n=2m$ be the real dimension.  Here $\langle \cdot, \cdot\rangle$ is the bilinear extension of the Riemannian product and the curvature $R$ follows the convention of \cite{Tian}.  Under this condition the authors derived the complex Hessian comparison theorem for the distance function  $\rho_{p}(x)$ to a fixed point $p$ of any K\"ahler manifold $M$ with the bisectional curvature bounded from below, with the corresponding  distance function of a complex space form with constant holomorphic sectional curvature $2\lambda$. (For the case $\lambda=0$, the result can also be derived from the Li-Yau-Hamilton type estimate for the heat kernel \cite{Cao-Ni}.)  The authors of \cite{LW}  also derived a diameter estimate (for $\lambda>0$)  as well as a volume comparison result. More recently, the volume comparison result was generalized to the distance function to a complex submanifold by Tam and Yu in \cite{TamYu}.  The reformulation in \cite{TamYu} seems stronger than the original one stated above by demanding (\ref{eq:lw-1}) on all $Z, W\in T'M$. In \cite{Liu},  the partial complex Hessian (only in the complex plane spanned by $\{\nabla \rho, C( \nabla \rho)\}$ with $C$ being the almost complex structure) comparison theorem was proved  under the assumption of the holomorphic sectional curvature is nonnegative. This result plays the crucial role \cite{Liu} in establishing the three-circle property for holomorphic functions on such K\"ahler manifolds. More recently, in \cite{XYang}, the projectivity was proved for compact K\"ahler manifolds with positive holomorphic sectional curvature. The common theme of the papers involving the comparison theorems  is that the  results were derived by applying the Bochner formula  to the length of the gradient $\|\nabla \rho\| (=1)$ in a similar spirit as the proofs of the Hessian and Laplacian comparison theorems in \cite{BZ,HK} (cf. also \cite{petersen}, where the Hessian comparison was made almost trivial for the case that the curvature is bounded from above), and \cite{GHL} respectively. Namely they are based on Ricatti's type inequality on the Hessian of $\rho$, or  the Bochner formula applying to $\|\nabla \rho\|$, instead of the more classical approach of Rauch via the comparison of the  index forms and Jacobi fields. On the other hand, the consideration via the second variation and the index forms has a lot of success in understanding the geometry and topology of the Riemannian manifolds. Even for K\"ahler manifolds with positive holomorphic sectional curvature there exists  the work of Tsukamoto on the diameter estimate and the simply connectedness \cite{Tsu}. Note that the diameter estimate of Li-Wang is a special result of Tsukamoto since the lower bound on the bisectional curvature posed by (\ref{eq:lw-1}) implies that the holomorphic sectional curvature $H(Z)=R(Z, \overline{Z}, Z, \overline{Z})\ge 2\lambda |Z|^4$. There are also Leftschez type theorems  for complex or Levi flat real submanifolds in a nonnegatively curved K\"ahler manifold  utilizing the index estimates of the energy functional such as \cite{Schoen-Wolfson} and \cite{Ni-Wolfson}.  Despite the work and the effort mentioned above the K\"ahler analogue of the sharp volume comparison (Bishop type) and the sharp diameter estimate (Bonnet-Myers type) are still elusive.  One goal of this paper is to apply the second variational/index form consideration to the K\"aher setting and prove several comparison and rigidity results generalizing some of  the results mentioned above with the hope of  bridging the gap between the Riemannian and K\"ahler setting. Another goal is to study a condition which is complementary to the holomorphic sectional curvature. Namely we shall study the comparison and vanishing theorems under  conditions on the {\it orthogonal Ricci curvature}. The vanishing theorem proved in this paper implies the projective embedding (namely projectivity of the underlying K\"ahler manifold) related to this curvature condition. This suggests that an algebraic geometric characterization of such K\"ahler manifolds perhaps is an interesting problem in view of the Fano varieties characterized via the Yau's solution to the Calabi conjecture.

Before getting into the statement of the result, we first recall various notions of  curvature for K\"ahler manifolds. If $Z=\frac{1}{\sqrt2} \left(X-\sqrt{-1} C(X)\right), W=\frac{1}{\sqrt2} \left(Y-\sqrt{-1} C(Y)\right)$, the first Bianchi gives the following expansion in terms of real vectors
$$
R(Z, \overline{Z}, W, \overline{W})= R(X, C(X), C(Y), Y)=R(X, Y, Y, X)+R(X, C(Y), C(Y), X).
$$

Besides the bisectional curvature, there exists the notion of the orthogonal bisectional curvature $R(Z, \overline{Z}, W, \overline{W})$ for any pair $Z, W$ with $\langle Z, \overline{W}\rangle=0$. Note that $\langle Z, \overline{W}\rangle =0$ means that $\langle X, Y\rangle=\langle X, C(Y)\rangle=0$. The holomorphic sectional curvature can be expressed as
 $$R(Z, \overline{Z}, Z, \overline{Z})=R(X, C(X), C(X), X).$$
For the sake of convenience in writing, we will sometimes use $H$ to denote the holomorphic sectional curvature and $B^{\perp}$ to denote the orthogonal bisectional curvature.

Clearly the lower bound on holomorphic sectional curvature,  on orthogonal bisectional curvature, or on bisectional curvature, are quite different assumptions. There are unitary symmetric metrics on $\mathbb{C}^m$ with nonnegative orthogonal bisectional curvature (abbreviated as (NOB)) but not with nonnegative bisectional curvature. There even exists algebraic K\"ahler curvature $R$ with nonnegative holomorphic sectional curvature and nonnegative orthogonal bisectional curvature, but not nonnegative bisectional curvature \cite{Tam}. There is also a weaker notion called {\it quadratic orthogonal bisectional curvature,} or {\it quadratic bisectional curvature} for short, denoted as $QB$, which is defined for any real vector $\{a_i\}_{i=1}^m$ and any unitary
 frame $\{E_i\}$ of $T'M$,
$QB(a)=\sum_{i, j}R_{i\bar{i}j\bar{j}}(a_i-a_j)^2$.  Its nonnegativity, abbreviated as (NQOB), means that $QB(a)\geq 0$ for any $a$ and any unitary frame $\{ E_i\}$. This curvature condition was formally introduced by Wu-Yau-Zheng in 2009 in \cite{WuYauZheng}, although it appeared implicitly in the work of Bishop of Goldberg \cite{BishopGoldberg} in 1965 already, where they showed that compact K\"ahler manifold with positive bisectional curvature must have its second Betti number equal to $1$. The  first example of compact K\"ahler manifold with (NQOB) but not (NOB) was established by Li-Wu-Zheng \cite{LiWuZheng} in 2013, and shortly after, Chau-Tam \cite{Chau-Tam} fully classified all (NQOB) K\"ahler C-spaces of classical types. See also \cite{NT2} for the role of (NQOB) in solving the Poincar\'e-Lelong equation. More recently in \cite{Ni-Niu2}, a gap theorem and a Liouville theorem were proved under (NOB) and nonnegative Ricci curvature.

Note that the nonnegative holomorphic sectional curvature does not imply the nonnegative Ricci curvature as shown by Hitchin's examples \cite{Hitchin}. On the other hand, there exists the notion of  {\it antiholomorphic Ricci curvature $Ric^{\perp}(X, X)$ }(for any real vector $X$, coined for example in \cite{Mi-Pal}, but no geometric implication of it was given) which is defined  as
$$
Ric^{\perp}(X, X)=\sum R(X, e_i, e_i, X)=Ric(X, X)-\frac{1}{|X|^2}R(X, C(X), C(X), X),
$$
where $\{e_i\}$ is any orthonormal frame of $\{X, C(X)\}^{\perp}$. In view of the above notions of (NOB) and (NQOB) it seems more sensible to called it {\it orthogonal Ricci curvature}. Let $E_i=\frac{1}{\sqrt2}(e_i-\sqrt{-1} C(e_i))$ be a unitary frame such that $e_1=\frac{X}{|X|}=\tilde{X}$. Following the convention $e_{n+i}=C(e_i)$, direct calculation shows that
\begin{eqnarray*}
\frac{1}{|X|^2}\Ric^{\perp}(X, X)&=&Ric^{\perp}(\tilde{X}, \tilde{X})=Ric(\tilde{X}, \tilde{X})-R(\tilde{X}, C(\tilde{X}), C(\tilde{X}), \tilde{X})\\
&=& Ric(E_1, \overline{E}_1)-R(E_1, \overline{E_1}, E_1, \overline{E}_1)=\sum_{j=2}^{m}R(E_1, \overline{E}_1, E_j, \overline{E}_j).
\end{eqnarray*}
Hence $Ric^{\perp}(\tilde{X}, \tilde{X})=Ric(E_1, \overline{E}_1)-R_{1\bar{1}1\bar{1}}$. Here we have also  used that $Ric(E_i,\overline{E}_i)=Ric(e_i, e_i)$. By Proposition 3.1 of \cite{HT} (see also \cite{Niu}), the nonnegativity of quadratic orthogonal bisectional curvature implies the nonnegativity of the orthogonal Ricci curvature. On the other hand, the example constructed there shows that there exist some unitary symmetric metrics on $\mathbb{C}^m$ such that the curvature has  nonnegative quadratic orthorgonal bisectional curvature (hence the orthogonal Ricci curvature is nonnegative), but the Ricci curvature is negative somewhere. In the later section of this paper we show examples of metrics with even nonnegative orthogonal bisectional curvature (which is stronger than the (NQOB)), but Ricci curvature, as well as the holomorphic sectional curvature,  can be negative somewhere. This shows that the $Ric^{\perp}(\cdot, \cdot)$ is a sensible notion for K\"ahler manifolds, and is different from the Ricci tensor. Nevertheless (NQOB) does imply the nonnegativity of the scalar curvature as shown in \cite{Chau-Tam, Niu}. In fact the non-negativity of $Ric^{\perp}$ also implies the nonnegativity of the scalar curvature from the following estimate.

\begin{lemma}\label{lemma-easy} The nonnegative orthogonal Ricci curvature implies the nonnegativity of the scalar curvature $S$. In fact there exists the following pointwise estimate for $m\ge 2$:
$$
S(y)\ge \frac{2m(m+1)}{m-1} \min_{Z\in \mathbb{S}^{2m-1}_y\subset T'M   } Ric^{\perp}(Z, \overline{Z}).
$$
\end{lemma}
Since the nonnegativity of the quadratic orthogonal bisectional curvature (NQOB) implies the nonnegativity of the orthogonal Ricci (abbreviated as $Ric^{\perp}\ge 0$),  Lemma \ref{lemma-easy} implies the result on the nonnegativity of the scalar curvature of \cite{Chau-Tam, Niu}).

Given any fixed point $p$, let $\rho(x)$ be the distance function to $p$. The Hessian of $\nabla^2 \rho(\cdot, \cdot)$ can be extended bi-linearly to $T_p^\mathbb{C} M$. Direct calculation shows that
$$
\nabla^2 \rho(E_i, \overline{E}_i)=\frac{1}{2}\left(\nabla^2 \rho( e_i, e_i)+\nabla^2 \rho(C(e_i), C(e_i))\right).
$$
This shows that $\Delta \rho=\sum_{i=1}^m \nabla^2 \rho((E_i, \overline{E}_i)=\frac{1}{2}\sum_{i=1}^m \left(\nabla^2 \rho( e_i, e_i)+\nabla^2 \rho(C(e_i), C(e_i))\right)$. Here $\{E_i\}$ is a unitary frame.
We define $\Delta^{\perp}$ the orthogonal Laplacian to be
$$
\Delta^{\perp} \rho=\Delta \rho-\nabla^2 \rho(Z, \overline{Z})
$$
where $Z=\frac{1}{\sqrt2}(\nabla \rho-\sqrt{-1}C(\nabla \rho))$. We call the last term {\it holomorphic Hessian } of $\rho$. The first comparison theorem we prove is on the orthogonal Laplacian assuming the orthogonal Ricci curvature comparison and the holomorphic Hessian comparison assuming the holomorphic sectional curvature comparison.

\begin{theorem}\label{thm-com1}(i) Let $(M^m, g)$ be a K\"ahler manifold with $Ric^{\perp}\ge (m-1) \lambda $. Let $(\tilde{M}, \tilde{g})$ be the complex space form with constant holomoprhic sectional curvature $2\lambda$. Let $\rho(x)$ be the distance function to a point $p$ (and $\tilde{\rho}$ be the corresponding distance function to a point $\tilde{p}$). Then for point $x$, which is not in the cut locus of $p$,
$$
\Delta^{\perp} \rho (x) \le \Delta^{\perp}\tilde{\rho}\left.\right|_{\tilde{\rho}=\rho(x)}=(m-1)\cot_{\frac{\lambda}{2}}(\rho).
$$
(ii)  Let $(M^m, g)$ be a K\"ahler manifold with holomorphic sectional curvature $H\ge 2\lambda $. Let $(\tilde{M}, \tilde{g})$ be the complex space form with constant holomoprhic sectional curvature $2\lambda$. Let $\rho(x)$ ($\tilde{\rho}$) be the distance function to a complex submanifold $P$ in $M$ ($\tilde{P}$ in $\tilde{M}$). Then for $x$ not in the focal locus of $P$,
$$
\left.\nabla^2\rho (Z, \overline{Z})\right|_{x}   \le \left. \nabla^2\tilde{\rho}(\tilde{Z}, \overline{\tilde{Z}})\right|_{ \tilde{\rho}=\rho(x)}.
$$
Here $Z=\frac{1}{\sqrt{2}}(\nabla \rho-\sqrt{-1} C(\nabla \rho))$, $\tilde{Z}=\frac{1}{\sqrt{2}}(\nabla \tilde{\rho}-\sqrt{-1} C(\nabla \tilde{\rho}))$. In particular, if $\lambda=0$ and $\tilde{P}$ is a point
$$
\left. \nabla^2\rho (Z, \overline{Z})\right|_{x} \le \frac{1}{2\rho(x)} \iff \nabla^2 \log \rho (Z, \overline{Z})\le 0.
$$
\end{theorem}

\begin{remark} The part (ii) was proved by G. Liu in \cite{Liu}  for the case that $P$ and $\tilde{P}$ are two points. The proof of \cite{Liu} follows the argument in \cite{LW}.
The results provide a generalization of the comparison theorem proved in \cite{LW}. Besides the point that the proof here uses a different argument, more importantly, the results signify the the geometric implications of orthogonal Ricci curvature and holomorphic sectional curvature.
\end{remark}

If both assumptions in (i) and (ii) are satisfied, the estimates in Theorem \ref{thm-com1} implies the volume comparison as in \cite{LW}.
\begin{corollary}\label{coro-wang} Assume that $(M^m, g)$ satisfies that
$Ric^{\perp}\ge (m-1) \lambda $ and $H\ge 2\lambda $. Then for any points $x\in M$, $\tilde{x}\in \tilde{M}$, $\Delta \rho (x) \le \left.\Delta \tilde{\rho}\right|_{\rho(x)}$,  and for any  $0<r\le R$,
$$
\frac{Vol(B(x, R))}{Vol(B(x, r))}\le \frac{Vol(\tilde{B}(\tilde{x}, R))}{Vol(\tilde{B}(\tilde{x}, r))} $$
where $\tilde{B}(\tilde{x}, R)$ is the ball in the complex space form. Equality holds if and only if $B(x, R)$ is holomorphic-isometric to the ball in the complex space form.
\end{corollary}
Note that the lower bounds of the  orthogonal Ricci and holomorphic sectional curvature implies the Ricci lower bound $Ric(X, X)\ge (m+1) |X|^2$. But the comparison in the K\"ahler case is sharper than the Riemannian setting.

The first part of Theorem \ref{thm-com1} can be generalized to the cases of complex hypersurfaces, which can be viewed as the K\"ahler version of Heintze-Karcher theorem \cite{HK} with the assumption on the Ricci curvature being replaced by the orthogonal Ricci.

\begin{theorem}\label{thm-com11} Let $(M^m, g)$ be a K\"ahler manifold with $Ric^{\perp}\ge (m-1) \lambda $. Let $(\tilde{M}, \tilde{g})$ be the complex space form with constant holomoprhic sectional curvature $2\lambda$. Let $\rho(x)$ be the distance function to a complex hypersurface $P$ (and $\tilde{\rho}$ be the corresponding distance function to a totally geodesic complex hypersurface $\tilde{P}$). Then for point $x$, which is not in the focal locus of $P$,
$$
\Delta^{\perp} \rho (x) \le \Delta^{\perp}\tilde{\rho}\left.\right|_{\tilde{\rho}=\rho(x)}=(m-1)\tan_{\frac{\lambda}{2}}(\rho).
$$
\end{theorem}
Note that $\tan_{\frac{\lambda}{2}}(t)$ is a little different from the conventional trigonometric function. In fact for $\lambda>0$,
$\tan_{\frac{\lambda}{2}}(t)=-\sqrt{\frac{\lambda}{2}}\cdot \frac{\sin(\sqrt{\frac{\lambda}{2}}t)}{\cos (\sqrt{\frac{\lambda}{2}}t)}$.
The above result strengthens that the sensible notion $Ric^{\perp}$ is related to the orthogonal Laplacian $\Delta^{\perp}$. If one assumes additionally the bound on the holomorphic sectional curvature, then one has the level hypersurface area comparison result similar to that of \cite{HK}, but sharper than the Riemannian setting due to K\"ahlerity.

Similarly one can consider the orthogonal Hessian of a real function $u$ to be $\nabla^2 u(Z, \overline{Z})$ restricted to the space consisting of all
$Z\perp \{ \nabla u, C(\nabla u)\}$. By now it is natural to infer that the orthogonal bisectional curvature gives comparison theorem for the orthogonal Hessian.

\begin{theorem}\label{thm-com2} Let $(M^m, g)$ be a K\"ahler manifold with $R(Z, \overline{Z}, W, \overline{W})\ge \lambda |Z|^2|W|^2$ for any $Z\perp W$ (namely the orthogonal bisectional curvature is bounded from the below by $\lambda$, which we abbreviate as $B^{\perp}\ge \lambda$). Let $(\tilde{M}, \tilde{g})$ be the complex space form with constant holomoprhic sectional curvature $2\lambda$. Let $\rho(x)$ be the distance function to a point $p$ (and $\tilde{\rho}$ be the corresponding distance function to a point $\tilde{p}$). Then for point $x$, which is not in the cut locus of $p$, restricted to the spaces of vectors $Z$ which are perpendicular to $\{ \nabla \rho, C(\nabla \rho)\}$ (as well as to $\{ \nabla \tilde{\rho}, C(\nabla \tilde{\rho})\}$)
$$
\nabla^2 \rho (x) \le \nabla^2\tilde{\rho}\left.\right|_{\tilde{\rho}=\rho(x)}.
$$
\end{theorem}

A similar argument as in the classical Bonnet-Myers theorem  implies that any complete K\"ahler manifold whose $Ric^{\perp}$ is bounded from below by a positive constant must be compact. This implies that any compact K\"ahler manifold with positive orthogonal Ricci curvature must have finite fundamental group.

For compact K\"ahler manifolds, in the following we will focus on the relation between  the holomorphic sectional curvature $H$, the Ricci curvature $Ric$, and the orthogonal Ricci curvature $Ric^{\perp }$. In terms of their strength, all three notions of curvature are sitting between bisectional curvature and scalar curvature, in the sense that when the bisectional curvature is positive, all three are positive, while when any one of them is positive, the scalar curvature is positive.

However, the relationship between these three curvature conditions is quite subtle, except the fact that $Ric= H + Ric^{\perp}$. By Yau's solution to the Calabi conjecture \cite{Yau}, compact K\"ahler manifolds with positive Ricci are exactly the projective manifolds with positive first Chern class, namely the Fano manifolds.

For compact K\"ahler manifolds with positive $H$, it was conjectured by Yau (cf. Problem 47,  \cite{YauP}), and recently proved by X. Yang \cite{XYang} that such manifolds are all projective. Hence by the recent work of Heier and Wong \cite{HeierWong} (see also \cite{XYang} for an alternative proof) they are all rationally-connected, meaning that any two points on the manifold can be joined by a rational curve. On the other hand, it was conjectured by Yau also that any rational or unirational manifold admits K\"ahler metrics with positive $H$. But this is far from being settled, as even on the surface ${\mathbb P}^2\# 2\overline{{\mathbb P}^2}$,  the blowing up of  ${\mathbb P}^2$ at two points, it is still an open question whether there exists such a metric.

It is certainly a natural question to understand the class of compact K\"ahler manifolds with positive $Ric^{\perp}$. We propose the following:

\begin{conjecture} Let $M^m$ ($m\geq 2$) be a compact K\"ahler manifold with $Ric^{\perp} >0$ everywhere. Then for any $1\leq p\leq m$, there is no non-trivial global holomorphic $p$-form, namely, the Hodge number $h^{p,0}=0$. In particular, $M^m$ is projective and simply-connected.
\end{conjecture}

Let us first explain the ``in particular" part in the above conjecture. Note that once we have the vanishing of $h^{p,0}$ for all $1\leq p\leq m$, then the vanishing of $h^{2,0}$ implies that $M^m$ is projective. Also, now since
$$ \chi({\mathcal O}_M) = 1 - h^{1,0} + h^{2,0} - \cdots + (-1)^m h^{m,0} =1,$$
where ${\mathcal O}_M$ is the structure sheaf, we know that such a manifold $M^m$ must be simply-connected since $\pi_1(M)$ is finite, and the Riemann-Roch theorem which asserts that the arithmetic genus $\chi$ is given as the integral over $M$ of a polynomial in Chern classes, as in \cite{Ko}.

We remark that for $M^m$ in the conjecture, $h^{1,0}=0$ since $\pi_1(M)$ is finite, and $h^{m,0}=0$ since $M^m$ has positive scalar curvature, thus the canonical line bundle cannot admit any non-trivial global holomorphic section. In fact, its Kodaira dimension must be $-\infty $ as it has positive total scalar curvature. So the conjecture is really about the cases $2\leq p\leq m-1$.

We also remark that, when $m=2$, the only compact K\"ahler surface with positive $Ric^{\perp}$ is (biholomorphic to) ${\mathbb P}^2$. This is because $Ric^{\perp }$ is equivalent to orthogonal bisectional curvature $B^{\perp}$ when $m=2$. By a result of Gu and Zhang \cite{GuZhang}, any compact, simply-connected K\"ahler manifold $M^m$ with positive $B^{\perp}$ is biholomorphic to ${\mathbb P}^m$ since the K\"ahler-Ricci flow takes any such metric into a metric with positive bisectional curvature (see also an alternate argument by Wilking in \cite{Wilking}). It would certainly be an interesting question to understand the class of threefolds or fourfolds with the $Ric^{\perp } >0$ condition. In this direction we prove the following partial result.

\begin{theorem}\label{thm-1connect} Let $M^m$ ($m\geq 2$) be a compact K\"ahler manifold with $Ric^{\perp} >0$ everywhere. Then its Hodge numbers $h^{m-1,0}=h^{2,0}=0$. In particular, $M^m$ is always projective. Also, it is simply-connected when $m\leq 4$.
\end{theorem}
In fact in Section 4 a stronger result is shown. Namely $h^{2, 0}=0$ (hence $M$ is projective) if the average of $Ric^\perp$ over two-planes is positive. An analogous result for $2$-scalar curvature was proved recently by authors \cite{Ni-Zheng2}.

For compact manifolds with $Ric^{\perp}<0$, one can obtain the following analogue of a result of Bochner \cite{wu}, which implies the finiteness of the automorphism group of such manifolds.

\begin{proposition}\label{bochner} Let $M^m$ be a compact K\"ahler manifold with $Ric^{\perp}<0$. Then there does not exists any nonzero holomorphic vector field.
\end{proposition}

It is an interesting question to find out whether or not such a manifold always admits a metric of negative Ricci curvature. That is, if its first Chern class is negative, or equivalently, if its canonical line bundle is ample.

Examples of  K\"ahler metrics  concerning various curvatures mentioned above and their relations can be found in sections 4-8. Among them we construct unitary complete K\"ahler metrics on $\mathbb{C}^m$ which have (NOB), positive Ricci, but negative holomorphic sectional curvature somewhere. This answers affirmatively a question raised recently in \cite{Ni-Niu2}.

\section{Proof of comparisons}

We first prove Lemma \ref{lemma-easy}. It is an easy consequence of a result of Berger.
\begin{proof} By a formula due to Berger, at any point $p\in M$, a K\"ahler manifold,
$$
S(p)= \frac{m(m+1)}{Vol(\mathbb{S}^{2m-1})}\int_{|Z|=1, Z\in T'_pM} H(Z)\, d\theta(Z).
$$
On the other hand it is easy to check that
$$
S(p)= \frac{2m}{Vol(\mathbb{S}^{2m-1})}\int_{|Z|=1, Z\in T_p'M} Ric(Z, \overline{Z})\, d\theta(Z).
$$
They imply that
\begin{equation}\label{eq:scalar-ricciperp}
\frac{m-1}{2m(m+1)}S(p)=\frac{1}{Vol(\mathbb{S}^{2m-1})}\int_{|Z|=1, Z\in T_p'M} Ric^{\perp}(Z, \overline{Z})\, d\theta(Z).
\end{equation}
The claimed result follows from (\ref{eq:scalar-ricciperp}) easily.
\end{proof}
 One can also   prove the following estimate on the holomorphic sectional curvature in terms of the orthogonal Ricci curvature.
\begin{corollary} \label{Gold} When $m\ge 2$, at any point $p$, for unitary $Z\in T_p'M$ with  $H_p(Z)=\max_{|W|=1} H_p(W)$,
$$
 H_p(Z)\ge \frac{2}{m-1} Ric^{\perp}(Z, \overline{Z}).
$$
In fact for any $W$ which is perpendicular to $Z$, $H_p(Z)\ge 2 R(Z, \overline{Z}, W, \overline{W}).$
Similarly if unitary $Z'\in T'_pM$ satisfying $H_p(Z')=\min_{|W|=1} H_p(W)$, then for unitary $W\perp Z'$
$$
H_p(Z')\le 2 R(Z', \overline{Z'}, W, \overline{W});\quad H_p(Z')\le \frac{2}{m-1} Ric^{\perp}(Z', \overline{Z'}).
$$
\end{corollary}
\begin{proof} For any complex number $a, b$ and $Z, W\in T'_pM$, it is easy to check that
\begin{eqnarray*}&\,&
H(aZ\!+\!bW)+H(aZ\!-\!bW)+H(aZ\!+\!\sqrt{\!-\!1} bW)+H(aZ\!-\!\sqrt{\!-\!1}bW)\\
&\,= &   4|a|^4 H(Z)
 +4 |b|^4 H(W)+16 |a|^2 |b|^2 R(Z, \overline{Z}, W, \overline{W}).
\end{eqnarray*}
For the unitary vector $Z$ and $W$ we choose $a, b$ such that $|a|^2+|b|^2=1$. Then if $Z$ attains the maximum of the holomorphic sectional curvature, for $W\perp Z$,
$$
4H(Z)\ge 4|a|^4 H(Z)
  +4 |b|^4 H(W)+16 |a|^2 |b|^2 R(Z, \overline{Z}, W, \overline{W}).
$$
The estimate $H(Z)\ge 2 R(Z, \overline{Z}, W, \overline{W})$ follows from the above. The claim on the orthogonal Ricci follows easily. For the minimal holomorphic sectional curvature, one can simply flip the above argument. A more direct approach is to consider function
$
f(\theta)=H(\cos \theta Z'+\sin \theta W)
$. The second derivative test applying to $f(\theta)$ and the one replacing $W$ by $\sqrt{\!-\!1}W$ implies the claimed estimate.
\end{proof}

 Before we prove the comparison theorem, let us recall some basics regarding the normal geodesics, the Jacobi fields with respect to a submanifold, the distance function and the tubular hypersurface with respect to a Riemannian submanifold  $P$ (only later we assume that $P$ is a complex submanifold).  Let $\mathcal{N}(P)$ denote the normal bundle of $P$. For any section $\nu(x)$ of the normal bundle  the exponential map $\exp_P$ can be defined as $\exp_x(\nu(x))$. First recall the concept of the {\it $P$-Jacobi field} along a {\it normal} geodesic $\gamma_u(\eta)$ with $u=\gamma'(0)\perp P$ at $p=\gamma(0)$. A Jacobi field $J(\eta)$ is called a $P$-Jacobi field along $P$ if it satisfies
$J(0)\in T_p P$ and $J'(0)-A_{\gamma'(0)} J(0)\perp T_p P$, where $A_{u}(\cdot)$ is the shape operator in the normal direction $u$. It is easy to check that if $\gamma(\eta, t)$ is a family of normal geodesics, with $\gamma(0, t)\in P$ and  $\frac{D \gamma}{\partial \eta} (0, t)\in T^{\perp}_{\gamma(0,t)}P$, $J(\eta)=\frac{D}{\partial t }\gamma(\eta, 0)$ is a $P$-Jacobi field. An elementary fact is that $\left. d\exp_P\right|_{\ell u}$ is degenerate if any only if there exists a non-zero $P$-Jacobi field $J(\eta)$ such that $J(\ell)=0$. The point $\gamma_u(\ell)$ is called a {\it focal point} (with respect to $P$).  The boundary operator $\frac{D J}{\partial \eta}-A_{\gamma'(0)} J(0)$ also arises from the second variation of the energy for a variation of pathes $\gamma(\eta, t)$ with the initial points in $P$ and a fixed end point:
$$
\left. \frac{d^2}{dt^2}\right|_{t=0} \mathcal{E}(\gamma)=\int_0^{\ell} |\nabla X|^2-\langle R(X, \gamma')\gamma', X\rangle \, d\eta +\langle A_{\gamma'(0)}(X(0)), X(0)\rangle
$$
with $X=\frac{D \gamma}{\partial t} (\eta, 0)$ being the tangent vector. Here $\mathcal{E}(\gamma)=\int_0^\ell |\frac{D \gamma}{\partial \eta}|^2\, d\eta$.  The polarization of the right hand side is called the index form. Namely the index form $I(X, Y)$ is given by
$$
I(X, Y)=\int_0^{\ell}\langle  \nabla X, \nabla Y\rangle -\langle R(X, \gamma')\gamma', Y\rangle \, d\eta  +\langle A_{\gamma'(0)}(X(0)), Y(0)\rangle-\langle A_{\gamma'(\ell)}(X(\ell)), Y(\ell)\rangle.
$$
Here the second boundary term enters only for the more general case that the ending points $\gamma(\ell, t)$ lying inside another submanifold $P'$. Allowing this flexibility is useful in \cite{Ni-Wolfson, Schoen-Wolfson}, but not needed when consider the distance function $\rho(x)$. We denote the index form (along $\gamma$) with $P'$ being a point as $I^{P}_{\gamma}(\cdot, \cdot)$ (otherwise we denote it as $I^{P, P'}_{\gamma}(\cdot, \cdot)$).

An easy but useful observation is the following relation between the Hessian of the distance function and the index form. Namely
\begin{equation}\label{eq-trans}
\left.\operatorname{Hessian}(\rho)\right|_{\rho(x)=\ell} (X, Y)= II_{\nabla \rho}(X, Y)=I^P_{\gamma [0, \ell]}(J_1, J_2)
\end{equation}
if $J_i(\eta)$ are $P$-Jacobi fields (in the case $P=\{p\}$ a point the assumption is equivalent to $J_i(0)=0$) and  $J_1(\ell)=X$, $J_2(\ell)=Y$. Here $II$ denotes the second fundamental form of hypersurface $\{x\, |\, \rho(x)=\ell\}$. In short {\it the Hessian of $\rho$, restricted to the subspace perpendicular to $\nabla \rho$, is the same as the index form, which  in turn is  the  same as the second fundamental form of the tubular hypersurface (of $P$) with respect to the unit exterior normal $\nabla \rho$.}

Another useful result is the index comparison lemma.

\begin{lemma}\label{lmm-index} Assume that $\gamma:[0, \ell]$ is a normal geodesic originated from $P$. Assume that there exists no focal point along $\gamma$. Let $X$ and $Y$ be two vector fields along $\gamma$ with $X$ being a $P$-Jacobi field, such that $ Y(0)\in T_{\gamma(0)}P$ and $X(\ell)=Y(\ell)$. Then
$$
I^P_{\gamma}(X, X)\le I^P_{\gamma}(Y, Y).
$$
The equality holds if any only if $Y(\eta)=X(\eta)$ for $\eta\in [0, \ell]$.
\end{lemma}
One can refer to \cite{Sakai} (cf. Chaper III, Lemma 2.10). In fact for any such $Y$, there exists a $P$-Jacobi field $X$ such that $Y(\ell)=X(\ell)$. An alternate proof is the following. First the index form can be used (replacing the Dirichlet energy) to define a Reilly quotient on the vector fields which are perpendicular to $\gamma'(\eta)$ and are tangent to the submanifolds (in the case $P'=\{x_0\}$, requiring vanishing boundary at $\gamma(\ell)$) at both ends. Then clearly the associated infinimum, namely the associated eigenvalue (which satisfies a Robin boundary condition at $\eta=0$ and Dirichlet condition at $\eta=\ell$)  is very positive for $\ell $ small. The positivity remains until a zero eigenvalue, namely a conjugate point (which is defined as when a non-zero eigenvector satisfying the Euler-Lagrange equation of the index form with suitable boundary condition, namely a $P$-Jacobi vector, can be obtained) is reached.

  For complex space form a useful lemma for this case is the following.
\begin{lemma} \label{lmm-csf}If $(\tilde{M}^m, g)$ is a K\"ahler manifold with constant holomorphic sectional curvature $2\lambda$. Let $n=2m$ be the real dimension and $\{\tilde{e}_i\}$ be a orthornormal  then
$$
R_{\tilde{e}_n, \tilde{e}_k} \tilde{e}_n =\frac{\lambda}{2} \tilde{e}_k, \mbox { if } \tilde{e}_k\perp \tilde{e}_n, C(\tilde{e}_n);  R_{\tilde{e}_n, \tilde{e}_k} \tilde{e}_n=2\lambda \tilde{e}_k, \mbox{ if } \tilde{e}_k=C(\tilde{e}_n).
$$
\end{lemma}
If one only wants a formula in the right hand side of the comparison, and does not care about the geometric meanings of the right hand side (such as in \cite{Liu}), one does not need the above lemma.

Now we can prove Theorem \ref{thm-com1}. Assume that $\gamma(\eta)$ and $\tilde{\gamma}(\eta)$ are two minimizing geodesics in $M$ and $\tilde{M}$. At $\gamma(\ell)$, let $\{e_i\}_{i=1}^{n=2m}$ be an orthonormal frame with $e_{2k}=C(e_{2k-1})$, and  $e_n=\nabla \rho$ and $e_{n-1}=-C(e_n)$ (namely $e_n=C(e_{n-1})$). By the definition,
$\Delta^{\perp} \rho=\frac{1}{2}\sum_{i=1}^{2m-2} \nabla^2 \rho( e_i, e_i)$. Let $\{\tilde{e}_i\}$ be the corresponding frame at $\tilde{\gamma}(\ell)$. Parallel transplant them along $\gamma$ and $\tilde{\gamma}$. By Lemma \ref{lmm-csf}, the Jacobi fields are given by $\tilde{J}_i(\eta)=\frac{S_{\frac12 \lambda} (\eta)}{S_{\frac12 \lambda}(\ell)} \tilde{e}_i(\eta)$ , for $1\le i\le 2m-2$, and $\tilde{J}_i(\eta)=\frac{S_{2 \lambda} (\eta)}{S_{2 \lambda}(\ell)} \tilde{e}_i(\eta)$ for $i=2m-1$. Here $$
S_\kappa(t) \doteqdot \left\{ \begin{matrix} \frac{1}{\sqrt{\kappa}}\sin \sqrt{\kappa} t, & \, \kappa>0,\cr
                                   t, &\, \kappa=0, \cr
                                   \frac{1}{\sqrt{|\kappa|}}\sinh \sqrt{|\kappa|}t, &\, \kappa<0;\end{matrix}\right.\quad S'_\kappa(t)\doteqdot \frac{d}{dt} S_\kappa(t); \quad \cot_\kappa(t)=\frac{S'_\kappa(t)}{S_\kappa(t)}.
                                   $$
Transplant $\{\tilde{J}_i(\eta)\}_{i=1}^{2m-2}$ along $\gamma(\eta)$ by letting $\overline{J}_i(\eta)= \frac{S_{\frac12 \lambda} (\eta)}{S_{\frac12 \lambda}(\ell)} e_i(\eta)$ we obtain $2m-2$ orthogonal vector fields along $\gamma(\eta)$ with $\overline{J}_i(\ell)=e_i(\ell)$ and $\overline{J}_i(0)=0$. Let $J_i(\eta)$ be the Jacobi fields with $J_i(\ell)=e_i$. Then
\begin{eqnarray*}
\left. 2\Delta^{\perp} \rho\right|_{\rho(x)=\ell}&=& \sum_{i=1}^{2m-2} \langle J'_i(\ell), J_i(\ell)\rangle= \sum_{i=1}^{2m-2}I_{\gamma[0, \ell]}(J_i, J_i);\\
\left. 2\Delta^{\perp} \tilde{\rho}\right|_{\tilde{\rho}(x)=\ell}&=& \sum_{i=1}^{2m-2} \langle \tilde{J}'_i(\ell), \tilde{J}_i(\ell)\rangle \sum_{i=1}^{2m-2}I_{\tilde{\gamma}[0, \ell]}(\tilde{J}_i, \tilde{J}_i).
\end{eqnarray*}
The curvature assumption, together with the initial conditions $J_i(0)=\tilde{J}_i(0)=0$,  implies that
$$
\sum_{i=1}^{2m-2}I_{\gamma[0, \ell]}(\overline{J}_i, \overline{J}_i)\le \sum_{i=1}^{2m-2}I_{\tilde{\gamma}[0, \ell]}(\tilde{J}_i, \tilde{J}_i).
$$
The result then follows from the index form comparison Lemma \ref{lmm-index}. This completes the proof on the comparison of $\Delta^{\perp} \rho$.

To prove the comparison on the complex Hessian, note that $\nabla^2 \rho(Z, \overline{Z}) =\frac{1}{2}\nabla^2 \rho (e_{n-1}, e_{n-1})$, where $Z=\frac{1}{\sqrt{2}}\left(\nabla \rho-\sqrt{-1}C(\nabla \rho)\right)$. Now let
$\overline{J}_{n-1}(\eta)=\frac{S_{2\lambda}(\eta)}{S_{2\lambda}(\ell)} e_{n-1}$ as before. It is easy to check that
$ \overline{J}_{n-1}(0)=0$ and $\overline{J}'_{n-1}(0)\perp T_{\gamma(0)}P$ (no need to check this for the previous case since $P=\{x_0\}$ being a point). Now the assumption on the holomorphic sectional curvature implies that
$$
I_{\gamma[0, \ell]}(\overline{J}_{n-1}, \overline{J}_{n-1})\le I_{\tilde{\gamma}[0, \ell]}(\tilde{J}_{n-1}, \tilde{J}_{n-1}).
$$
 The claimed result again follows from the index form comparison Lemma \ref{lmm-index}.

\section{Extensions}

First we prove the Theorems \ref{thm-com11} and \ref{thm-com2}. The proof of Theorem \ref{thm-com2} follows verbatim as the proof of Theorem \ref{thm-com1}. For Theorem \ref{thm-com11}, we  construct of the vector fields $\{\overline{J}_i\}$ satisfying  different boundary conditions at $\eta=0$.  First we define
$$
C_\kappa(t) \doteqdot \left\{ \begin{matrix} \frac{1}{\sqrt{\kappa}}\cos \sqrt{\kappa} t, & \, \kappa>0,\cr
                                   1, &\, \kappa=0, \cr
                                   \frac{1}{\sqrt{|\kappa|}}\cosh \sqrt{|\kappa|}t, &\, \kappa<0;\end{matrix}\right.\quad C'_\kappa(t)\doteqdot \frac{d}{dt} C_\kappa(t); \quad  \tan_\kappa(t)=\frac{C'_\kappa(t)}{C_\kappa(t)}.
                                   $$
Now we let $\overline{J}_i(\eta)=\frac{C_{\frac{\lambda}{2}}(\eta)}{C_{\frac{\lambda}{2}}(\ell)} e_i(\eta)$.   Since at $\eta=0$, $e_n(0)=\gamma'(0)$ and $e_{n-1}(0)=C(\gamma'(0))$ are perpendicular to $P$, $\{e_{i}(0)\}_{i=1}^{2m-2}$ are tangent to $P$. Since $P$ is minimal
$$
\sum_{i=1}^{2m-2} \langle A_{\gamma'(0)}(\overline{J}_i(0)), \overline{J}_i(0)\rangle =0.
$$
Hence (if we adapt the Einstein convention)
$$
\sum_{i=1}^{2m-2}I_{\gamma[0, \ell]}(\overline{J}_{i}, \overline{J}_{i})= \int_0^\ell \|\overline{J}'_{i}\|^2-\langle R_{\overline{J}_i, \gamma'}\gamma', \overline{J}_i\rangle =\frac{1}{C^2_{\frac{\lambda}{2}}(\ell)} \int_0^\ell (2m-2)(C'_{\frac{\lambda}{2}})^2-C^2_{\frac{\lambda}{2}}Ric^{\perp}(\gamma', \gamma') .
$$
Then Theorem \ref{thm-com11}  follows from the index comparison Lemma \ref{lmm-index} and direct calculation of the right hand above (for $\sum_{i=1}^{2m-2}I_{\tilde{\gamma}[0, \ell]}(\tilde{J}_{i}, \tilde{J}_{i})$).

The argument above  can be extended to the case that $P$ is a Levi-flat real hypersurface, observing that the boundary term vanishes due to the Levi-flatness (cf. \cite{Ni-Wolfson}).

\begin{corollary}\label{coro-levi}Let $(M^m, g)$ be a K\"ahler manifold with $Ric^{\perp}(X, X)\ge (m-1) \lambda |X|^2$. Let $(\tilde{M}, \tilde{g})$ be the complex space form with constant holomoprhic sectional curvature $2\lambda$. Let $\rho(x)$ be the distance function to a real Levi flat  hypersurface $P$ (and $\tilde{\rho}$ be the corresponding distance function to a totally geodesic complex hypersurface $\tilde{P}$). Then for point $x$, which is not in the focal locus of $P$,
$$
\Delta^{\perp} \rho (x) \le \Delta^{\perp}\tilde{\rho}\left.\right|_{\tilde{\rho}=\rho(x)}=(m-1)\tan_{\frac{\lambda}{2}}(\rho).
$$
\end{corollary}

In \cite{Tsu}, it was proved that if a K\"ahler manifold $(M^m, g)$ has positive lower bound $2\lambda$ on its holomorphic sectional curvature, then it must be compact with diameter bounded from above by $\frac{\pi}{\sqrt{2\lambda}}$. The following generalizes this slightly.

\begin{proposition}\label{prop-index} Let $(M^m, g)$ be a compact K\"ahler manifold with holomorphic sectional curvature bounded from below by $2\lambda>0$. Then for any geodesic $\gamma(\eta): [0, \ell]\to M$ with length $\ell>\frac{\pi}{\sqrt{2\lambda}}$, the index $i(\gamma)\ge 1$.
\end{proposition}
\begin{proof} Let $e_{n-1}(\eta)=C(\gamma'(s))$. Let $X(\eta)=\sin\left(\frac{\pi}{\ell}\eta\right)e_{n-1}(\eta)$.
Then
\begin{eqnarray*}
I(X, X)&=&\int_0^\ell \left(\frac{\pi}{\ell}\right)^2 \cos^2\left(\frac{\pi}{\ell}\eta\right)-\sin^2 \left(\frac{\pi}{\ell}\eta\right)\langle R_{C(\gamma'), \gamma'}\gamma', C(\gamma')\rangle\\
&\le &\left(\frac{\pi}{\ell}\right)^2\int_0^\ell \cos^2\left(\frac{\pi}{\ell}\eta\right)-2\lambda \int_0^\ell \sin^2 \left(\frac{\pi}{\ell}\eta\right)<0.
\end{eqnarray*}
This proves the claim.
\end{proof}
Moreover it was also proved in \cite{Tsu} that $M$ must be simply-connected.  The following is a generalization on the simply-connectedness.

\begin{proposition}\label{prop-tran} Let $(M^m, g)$ be a compact K\"ahler manifold with positive holomorphic sectional curvature. Then any holomorphic isometry of $M$ must have at least one fixed point.
\end{proposition}
\begin{proof} Assume that there exists such a map $\phi: M\to M$ with no fixed point. Then there exists $p$ such that $d(p, \phi(p))=\min_{q\in M} d(q, \phi(q))$. Let $\gamma$ be the minimal geodesic joining $p$ to $\phi(p)$ with $\ell$ being the length. First observe that $d\phi(\gamma'(0))=\gamma'(\ell)$. This follows from the triangle inequality and the estimate:
$$
d(\gamma(\eta), \phi(\gamma(\eta)))\le d(\gamma(\eta), \phi(p))+d(\phi(p), \phi(\gamma(\eta)))=d(\gamma(\eta), \phi(p))+d(p, \gamma(\eta))=d(p, \phi(p)).
$$
Now let $e_n=\gamma'(\eta)$. Let $e_{n-1}=C(e_n)$. Clearly $e_{n-1}(\eta)$ is parallel. On the other hand, $e_{n-1}(0)=C(\gamma'(0))$, $e_{n-1}(\ell)=C(\gamma'(\ell))=C(d\phi(\gamma'(0)))=d\phi(e_{n-1})$. This shows that if $\beta(s)$ is a geodesic starting from $p$ with $\beta'(0)=e_{n-1}$, $\tilde{\beta}(s)=\phi(\beta(s))$ will be a geodesic starting from $\gamma(\ell)$ with $\tilde{\beta}'(0)=e_{n-1}(\ell)$. Consider the  variation
$\gamma(\eta, s)=\exp_{\gamma(\eta)}(se_{n-1}(\eta))$. The second variation formula on the energy $\mathcal{E}(s)=\frac{1}{2}\int_0^\ell |\frac{\partial \gamma}{\partial \eta}|^2$ gives that
$$
\frac{d^2}{ds^2} \mathcal{E}(0)=-\int_0^\ell \langle R_{e_{n-1}, \gamma'}\gamma', e_{n-1}\rangle <0.
$$
This contradicts to  that $\gamma_0(\eta)=\gamma(\eta, 0)$ is length minimizing (hence also energy minimizing) among all $\gamma_s(\eta)=\gamma(\eta, s)$, which joins $\beta(s)$ to $\tilde{\beta}(s)=\phi(\beta(s))$. \end{proof}

Regarding to the diameter estimate we have the following result under the assumption of the orthogonal Ricci lower bound, whose proof is the same as that of Myers' theorem.

\begin{theorem}\label{thm-diameter}
Let $(M^m, g)$ be a K\"ahler manifold with $Ric^{\perp}(X, X)\ge (m-1) \lambda |X|^2$ with $\lambda>0$. Then $M$ is compact with diameter bounded from the above by $\sqrt{\frac{2}{\lambda}}\cdot \pi$. Moreover, for any geodesic $\gamma(\eta): [0, \ell]\to M$ with length $\ell>\sqrt{\frac{2}{\lambda}}\cdot \pi$, the index $i(\gamma)\ge 1$.
\end{theorem}

Note that this estimate is not sharp for Fubini-Study metrics. It is an interesting question {\it whether or not a compact K\"ahler manifold with positive orthogonal Ricci curvature is simply-connected.} The case for Ricci curvature is a theorem of S.                Kobayashi \cite{Ko}. The following result provides a generalization of a result of Tam and Yu \cite{TamYu}.

\begin{corollary}Assume that $(M^m, g)$ satisfies that
$Ric^{\perp}(X, X)\ge (m-1) \lambda |X|^2$ and $H(X)\ge 2\lambda |X|^4$ with $\lambda>0$. Assume that there exists a complex hypersurface $P$ and a point $Q\in M$ such that $d(P, Q)=\frac{\pi}{\sqrt{2\lambda}}$. Then $(M^m, g)$ is holomorphic-isometric to a complex projective space with Fubini-Study metric.
\end{corollary}
\begin{proof} Without the loss of generality we let $\lambda=2$. Under the assumption, it is known that $d(P, Q)\le \frac{\pi}{2}$.  The assumption and the comparison theorems proved above implies that the area element with respect to level circle of a complex hypersurface over the area element of the level circle  of $\mathbb{CP}^{m-1}\subset \mathbb{CP}^m$, and the area element with respect to the level spheres (to a point) over that of sphere in $\mathbb{CP}^m$ are all monotone decreasing. This shows that for any $\ell \in (0, \frac{\pi}{2})$, $B(P, \ell)\cap B(Q, \frac{\pi}{2}-\ell)=\emptyset$ and
\begin{eqnarray*}
1 &\ge& \frac{Vol (B(P,\ell))}{Vol (M)}+\frac{Vol(B(Q, \frac{\pi}{2} -\ell))}{Vol(M)}\\
&\ge& \frac{1}{Vol (\mathbb{CP}^m)}\left(\int_{\mathbb{CP}^{m-1}}\int_0^\ell 2\pi \cos ^{2m-1} t \cdot \sin t\, dt+
 \int_{\mathbb{S}^{2m-1}}\int_0^{\frac{\pi}{2}-\ell} \sin^{2m-1} t \cdot \cos t\, dt\right)\\
 &=&1.
\end{eqnarray*}
The claimed rigidity follows from the equality case in the volume/area comparison as classical case.\end{proof}

\section{Proof of the vanishing theorem}

In this section we shall prove Theorem \ref{thm-1connect}.
In a recent paper \cite{XYang}, X. Yang proved that any compact K\"ahler manifold $M^m$ with positive holomorphic sectional curvature $H$  satisfies $h^{p,0}=0$ for all $1\leq p\leq m$, using the form version of the Bochner identity. By employing this method we prove that, under the $Ric^{\perp } >0$ condition,  $h^{m-1,0}=h^{2,0}=0$.

Let $s$ be a global holomorphic $p$-form on $M^m$. The Bochner identity (cf. Ch III, Proposition 1.5 of \cite{Ko2}, as well as Porposition 2.1 of \cite{Ni-JDG}) gives
$$  \partial \overline{\partial } |s|^2 = \langle \nabla s , \overline{\nabla s} \rangle  - \widetilde{R}(s, \overline{s}, \cdot , \cdot ) $$
where $\widetilde{R}$ stands for the curvature of the Hermitian bundle $\bigwedge^p\Omega$, and $\Omega=(T'M)^*$ is the holomorphic cotangent bundle of $M$. The metric on $\bigwedge^p\Omega$ is derived from the metric of $M^m$.

It is  useful to note that $\tilde{R}$ acts on $(p, 0)$ forms as special case of  the curvature action on tensors. Precisely we have the following formula for any holomorphic $(p, 0)$-form $s$ and any given tangent direction $v$ at the point $x_0$, namely, there will be local frame $\{dz_i\}$ which is unitary at a point $x_0$, such that
\begin{equation}\label{eq:40}
\langle \sqrt{-1}\partial\bar{\partial} |s|^2, \frac{1}{\sqrt{-1}}v\wedge \bar{v}\rangle =\langle \nabla_v s, \bar{\nabla}_{\bar{v}} \bar{s}\rangle +\frac{1}{p!}\sum_{I_p} \sum_{k=1}^p R_{v\bar{v}i_k \bar{i}_k}|a_{I_p}|^2,
\end{equation}
where $s=\frac{1}{p!}\sum_{I_p} a_{I_p}dz^{i_1}\wedge \cdots \wedge dz^{i_p}$ and $I_p=(i_1, \cdots, i_p)$. The $\langle\cdot , \cdot\rangle$ in the left hand side is the scalar product between the $(1, 1)$-forms and their dual, instead of bilinear extension of the Hermitian product. If $M$ admits metric of positive holomorphic section curvature, the second variation argument as in the proof of Corollary \ref{Gold} implies that
$R_{v\bar{v}i_k \bar{i}_k}>0$ for $v$, a unit vector which attains the minimum of the holomorphic sectional curvature among all unit vector $w\in T_{x_0}'M$ at the given $x_0$. This is the argument of \cite{XYang} proving the vanishing of $h^{p, 0}$ under the positivity of the holomorphic sectional curvature.

Now we adapt this to prove Theorem \ref{thm-1connect}.
If $s$ is not identically zero, then $|s|^2$ will attain its nonzero  maximum somewhere, say $x_0$,  and at this point we have
\begin{equation*}
 \widetilde{R}(s, \overline{s},v, \overline{v}) \geq 0
 \end{equation*}
for any type $(1,0)$ tangent vector $v\in T_{x_0}M$. We want to show that this will contradict the assumption $Ric^{\perp} >0$ when either $p=m-1$ or  $p=2$.

The $p=m-1$ case is easy. In a small neighborhood of $x_0$, we can write $s= f \varphi_2\wedge \cdots \wedge \varphi_m$, where $f\neq 0$ is a function and $\{ \varphi_1, \varphi_2, \ldots , \varphi_m\} $ are local $(1,0)$-forms forming a coframe dual to a local  tangent frame $\{E_1, \ldots , E_m\}$, which is unitary at $x_0$. Since
$$ \widetilde{R} (s, \overline{s},v, \overline{v}) = - |f|^2 \sum_{i=2}^m R_{i\overline{i}v\overline{v}} \geq 0  $$
for any tangent direction $v$, where $R$ is the curvature tensor of $M$. If we take $v=E_1$, we would get $Ric^{\perp} (E_1, \overline{E}_1) \leq 0$, a contradiction.

Now consider  the $p=2$ case. Suppose that $s$ is a non-trivial global holomorphic $2$-form on $M^m$. Let $r\geq 1$ be the largest positive integer such that the wedge product $s^r$ is not identically zero. Since we already have $h^{m,0}=h^{m-1,0}=0$, we know that $2r\leq m-2$.

We will apply the $\partial\bar{\partial}$-Bochner formula to the $2r$-form $\sigma = s^r$. Let $x_0$ be a maximum point of $|\sigma |^2$. At $x_0$, let us write  $s=\sum_{i,j} f_{ij}\varphi_i\wedge \varphi_j$ under any unitary coframe $\{\varphi_j\}$ which is dual to a local unitary tangent frame $\{E_j\}$. The $m\times m$ matrix  $A= (f_{ij})$ is  skew-symmetric. As is well-known (cf. \cite{Hua}), there exists unitary matrix $U$ such that $\ ^t\!U AU$ is in the block diagonal form where each non-zero diagonal block is a constant multiple of $E$, where
$$ E=\left[ \begin{array}{cc} 0 & 1 \\ -1 & 0 \end{array} \right]. $$
In other words, we can choose a unitary coframe $\varphi$ at $x_0$ such that
$$ s = \lambda_1 \varphi_1\wedge \varphi_2 + \lambda_2 \varphi_3 \wedge \varphi_4 + \cdots + \lambda_k \varphi_{2k-1}\wedge \varphi_{2k}, $$
where $k$ is a positive integer and each $\lambda_i\neq 0$. Clearly, $k\leq r$ since $s^k\neq 0$ at $x_0$. If $k<r$, then $\sigma = s^r =0$ at $x_0$, which is a maximum point for $|\sigma|^2$, implying $\sigma \equiv 0$, a contradiction. So we must have $k=r$. Thus $\sigma = \lambda \varphi_1\wedge \cdots \wedge \varphi_{2r}$, where $\lambda = \lambda_1 \cdots \lambda_k\neq 0$. From the Bochner formula, we get that
\begin{equation}\label{eq:41} \sum_{i=1}^{2r} R_{i\overline{i}v\overline{v}} \leq 0
\end{equation}
for any tangent direction $v$ of type $(1,0)$ at $x_0$. From this we shall  derive a contradiction to our assumption that $Ric^{\perp}>0$.

Denote by $W\cong {\mathbb C}^{2r}$ the subspace in $T_{x_0}'M$ spanned by $E_1, \ldots , E_{2r}$. By letting $v\in W$, we see that the `Ricci' of the restriction $R|_W$ of the curvature tensor $R$ on $W$ is nonpositive, thus the `scalar' curvature of $R|_W$ is also nonpositive:
\begin{equation}\label{eq:42}
 S|_W = \sum_{i,j=1}^{2r} R_{i\overline{i}j\overline{j}} \leq 0.
 \end{equation}
On the other hand, for each $1\leq j\leq 2r$, $Ric^{\perp}(E_j, \overline{E}_j)>0$. By adding them up, we get
\begin{equation}\label{eq:43}
0< \sum_{j=1}^{2r} Ric^{\perp}(E_j, \overline{E}_j) = \sum_{1\leq i\neq j \leq 2r} R_{i\overline{i}j\overline{j}} + \sum_{j=1}^{2r} \sum_{\ell =2r+1}^{m} R_{j\overline{j}\ell \overline{\ell }}.
\end{equation}
By applying (\ref{eq:41}) to $v=E_{\ell}$ for each $\ell$, we know that the second term on the right hand side of (\ref{eq:43}) is nonpositive, therefore we get
\begin{equation}\label{eq:44}
 \sum_{1\leq i\neq j \leq 2r} R_{i\overline{i}j\overline{j}} = S|_W - \sum_{i=1}^{2r} H(E_i) > 0.
\end{equation}
Note that for any $P\in U(2r)$, if we replace $\{ E_1, \ldots , E_{2r} \}$ by $\{ \tilde{E}_1, \ldots , \tilde{E}_{2r} \}$ where $\tilde{E}_i = P_{ij} E_j$, then the above inequality still holds. Taking the average integral $\aint \ $ over $U(2r)$, and using Berger's lemma, we get
$$ 0< S|_W - 2r \aint H(PE_1)  = S|_W - 2r \frac{2}{2r(2r+1)} S|_W = \frac{2r-1}{2r+1} S|_W, $$
so $S|_W>0$, contradicting (4.3). This proves that $h^{2,0}=0$ for any compact K\"ahler manifold $M^m$ with $Ric^{\perp}>0$ everywhere, and we have completed the proof of Theorem \ref{thm-1connect}.

The proof in fact yields the following more general result, which is in the same spirit of the result in \cite{Ni-Zheng2}.

\begin{corollary}\label{coro:31} The vanishing of Hodge number $h^{2,0}(M)$ follows if $(M, g)$ is compact and for any unitary pair $\{E_i\}_{i=1, 2}$ with $E_1\perp E_2$
$$
Ric^{\perp}(E_1, \overline{E}_1)+Ric^{\perp}(E_2, \overline{E}_2)>0.
$$
In particular, $M$ is projective.
\end{corollary}

Modifying the argument also proves the following result which in fact is different from the above corollary since $Ric^\perp(Z, \overline{Z})$ does not come from a Hermitian symmetric sesquilinear form. Similar to \cite{Ni-Zheng2}, for any $k$-subspace $\Sigma\subset T_x'M$,  we define
$$
Ric^{\perp}_k(x, \Sigma)=\aint_{Z\in \Sigma, |Z|=1} Ric^\perp(Z, \overline{Z})\, d\theta(Z)
$$
where $\aint f (Z)\, d\theta(Z)$ denotes $\frac{1}{Vol(\mathbb{S}^{2k-1})}\int_{\mathbb{S}^{2k-1}}f(Z)\, d\theta(Z)$.
We say $Ric_k^{\perp}(x)>0$ if for any $k$-subspace $\Sigma\subset T_x'M$, $Ric^{\perp}_k(x, \Sigma)>0$.
\begin{theorem}
Let $(M, g)$ be a compact K\"ahler manifolds such that $Ric^{\perp}_2(x)>0$ for any $x\in M$. Then
$h^{2, 0}=0$. In particular, $M$ is projective.
\end{theorem}
\begin{proof} First it is easy to see that $Ric_l^{\perp}(x)>0$ implies that $Ric_k^{\perp}(x)>0$ for any $k\ge l$. We observe that, if $\Sigma=\operatorname{span}\{E_1, E_2, \cdots, E_l\}$,
\begin{eqnarray*}
Ric^{\perp}_l(x, \Sigma)&=&\aint_{Z\in \Sigma, |Z|=1} Ric^\perp(Z, \overline{Z})\, d\theta(Z)=\aint_{Z\in \Sigma, |Z|=1} Ric(Z, \overline{Z})-H(Z)\, d\theta(Z)\\
&=&\aint \frac{1}{Vol(\mathbb{S}^{2m-1})}\int_{\mathbb{S}^{2m-1}} mR(Z, \overline{Z}, W, \overline{W})-H(Z)\, d\theta(W)\, d\theta(Z)\\
&=&\frac{1}{Vol(\mathbb{S}^{2m-1})}\int_{\mathbb{S}^{2m-1}}\left( \aint mR(Z, \overline{Z}, W, \overline{W})-H(Z)\, d\theta(Z)\right)d\theta(W)\\
&=& \frac{1}{l}\left(Ric(E_1, \overline{E}_1)+Ric(E_2, \overline{E}_2)+\cdots +Ric(E_l, \overline{E}_l) \right)-\frac{2}{l(l+1)}S_l(x, \Sigma)
\end{eqnarray*}
where $S_l(x, \Sigma)$ is the scalar curvature of $R$ restricted to $\Sigma$ (cf. \cite{Ni-Zheng2}).
Now we adapt the proof of Theorem \ref{thm-1connect} above, for $W=\operatorname{span}\{E_1, \cdots, E_{2r}\}$, the $\partial\bar{\partial}$-Bochner formula implies that $S_{2r}(x_0, W)\le 0$. On the other hand the above calculation and (\ref{eq:42}) implies that
$$
\frac{2r-1}{2r(2r+1)}S_{2r}(x_0, W)\ge Ric^{\perp}_{2r}(x_0,W)>0.
$$
The contradiction implies the theorem.
\end{proof}

Note that it is well known that (cf. \cite{Ko2}, Theorem 3.4 of Ch. 3) if $$Ric_k(x)=\min{\{E_i\}} \left(Ric(E_1, \overline{E}_1)+Ric(E_2, \overline{E}_2)+\cdots +Ric(E_k, \overline{E}_k) \right)>0$$ everywhere $h^{p, 0}=0$ for any $p\ge k$. It was recently  proved in \cite{Ni-Zheng2} that the same result holds if $S_k>0$. Given the above relation between $Ric^\perp_k(x)$, $Ric_k(x)$, and $S_k(x)$, it is natural to  conjecture that $h^{p, 0}=0$ if $Ric^\perp_k(x)>0$. Clearly an affirmative answer to this question  would imply the main conjecture in the introduction.

To prove Proposition \ref{bochner}, observe that for any holomorphic vector field $s$ the $\partial\bar{\partial}$-Bochner formula can be applied to obtain that
$$
\langle \sqrt{-1}\partial\bar{\partial} |s|^2, \frac{1}{\sqrt{-1}}v\wedge \bar{w}\rangle \ = \ \langle \nabla_v s, \bar{\nabla}_{\bar{w}} \bar{s}\rangle- R_{v\bar{w}s\bar{s}}.
$$
If $s$ is nonzero, as before at the point $x_0$, where $|s|^2$ attains its maximum we have that
$$
R_{v\bar{v} s\bar{s}}\ge 0
$$
for any direction $v$. Summing over an unitary basis of $\{s\}^{\perp}$ we have a contradiction with $Ric^{\perp}<0$.

\vspace{0.2in}

 Next we examine the correlation between the positivity of the three curvatures: $Ric$, $Ric^{\perp}$, and $H$. First of all, we observe that the positivity of two of them does not imply that of the third one, except the obvious case caused by $Ric = Ric^{\perp }+H$.

$ \bullet$ Examples with $Ric>0$, $H>0$ but $Ric^{\perp}\ngeq 0$.

To see such an example, let us consider the surface $M^2={\mathbb P}^2 \# \overline{{\mathbb P}^2}$, the blowing up of ${\mathbb P}^2$ at one point. We have
$$ M^2= \{ ([z_0:z_1:z_2], [w_1:w_2]) \in {\mathbb P}^2 \times {\mathbb P}^1 \mid z_1w_2=z_2w_1\} .$$
For $\lambda >0$, let $\omega_{\lambda}$ be the metric on $M^2$ which is the restriction of
$$ \sqrt{-1} \partial \overline{\partial } \log (|z_0|^2+ |z_1|^2+|z_2|^2) + \lambda \sqrt{-1} \partial \overline{\partial } \log (|w_1|^2+|w_2|^2) $$
on the product manifold ${\mathbb P}^2 \times {\mathbb P}^1$. By a straight forward computation, which will be included in the appendix, we will  show that $Ric>0$ everywhere if and only if $\lambda>\frac{1}{2}$, and $H>0$ everywhere if and only if $\lambda >1$. So for any $\lambda >1$, we get an example with the desired curvature condition. Note that the metric has $Ric^{\perp} \ngeq 0$. In fact, $M^2$ does not admit any K\"ahler metric with $Ric^{\perp }\geq 0$ everywhere by the result of Gu-Zhang \cite{GuZhang}.

$ \bullet$ Examples with $Ric>0$, $Ric^{\perp}>0$ but $H\ngeq 0$.

In later sections, we will construct examples of complete $\mathsf{U}(m)$-invariant K\"ahler metrics on ${\mathbb C}^m$, such that its Ricci curvature and orthogonal bisectional curvature $B^{\perp}$ are both everywhere positive, yet the holomorphic sectional curvature $H\ngeq 0$. In fact, there are such examples where $H$ is negative in some tangent directions at every point outside a compact subset. Note that as we mentioned before, $B^{\perp}>0\Rightarrow QB>0\Rightarrow Ric^{\perp} >0$.

We would also point out that, there are examples of K\"ahler metrics where one of these three curvature is positive, while the other two are not.

$ \bullet$ Examples with $H>0$ but $Ric\ngeq 0$, $Ric^{\perp}\ngeq 0$.

For instance, consider the Hirzebruch surface $F_n={\mathbb P}( {\mathcal O}_{{\mathbb P}^1} \oplus {\mathcal O}_{{\mathbb P}^1}(-n))$ with $n>2$. By a well-known result of Hitchin \cite{Hitchin}, all $F_n$ admit K\"ahler metric with $H>0$ everywhere. On the other hand, when $n> 2$, the first Chern class $c_1(F_n) \ngeq 0$, so there is no K\"ahler metric with $Ric\geq 0$. There is no K\"ahler metric with $Ric^{\perp} \geq 0$ either,  by the result of Gu-Zhang.

$ \bullet$ Examples with $Ric>0$ but $Ric^{\perp} \ngeq 0$, $H\ngeq 0$.

To see such an example, we can simply take the previous example (see the Appendix for details) of metric on ${\mathbb P}^2 \# \overline{{\mathbb P}^2}$ with the parameter $\lambda $ in $(\frac{1}{2}, 1)$. In this case one has $Ric >0$ everywhere, but $H$ is negative somewhere in some directions. The surface does not admit any metric with $Ric^{\perp}\geq 0$ for the reason given above.

Note that on this surface, there are metrics with $H>0$. In fact, it is conjectured by Yau that any rational surface (or any rational manifolds in higher dimensions) admits K\"ahler metric with $H>0$ everywhere, although this is still open for most of the rational surfaces.

$ \bullet$ Examples with $Ric^{\perp } >0$ but $Ric \ngeq 0$, $H\ngeq 0$.

  Examples of complete $\mathsf{U}(m)$-invariant K\"ahler metrics on ${\mathbb C}^m$ with the above curvature properties will be constructed in a later section. In fact the metric constructed will have $B^\perp>0$, but $Ric<0$ and $H<0$ for some directions at every point outside a compact subset.

\section{Examples--preliminary}

We will follow the notations of \cite{WZ}, \cite{HT}, and \cite{Ni-Niu2}. Let $g$ be a $\mathsf{U}(m)$-invariant K\"ahler metric on ${\mathbb C}^m$ with K\"ahler form $\omega_g$. Denote by $(z_1, \ldots , z_m)$ the standard holomorphic coordinates of ${\mathbb C}^m$ and write $r=|z_1|^2+ \cdots + |z_m|^2$. Since $g$ is $\mathsf{U}(m)$-invariant, one can write $\omega_g = \sqrt{\!-\!1}\partial \overline{\partial }P(r)$ for some smooth function $P$ on $[0,\infty )$. Note that $\omega_g >0$ means that the smooth functions $f=P'>0$ and $h=(rf)'>0$, and the metric is complete if and only if
$$ \int_0^{\infty } \sqrt{ \frac{h}{r} } \ dr = \infty . $$

Here we adapt the constructions in \cite{WZ, HT, Ni-Niu2} to illustrate unitary symmetric metrics on $\mathbb{C}^m$ with various properties promised in the last section. The basic is the ansatz and computation laid out in \cite{WZ}. Below is a summary.

In \cite{WZ}, Wu and Zheng  considered the $\mathsf{U}(m)$-invariant K\"ahler metrics on $\mathbb{C}^m$ and obtained necessary and sufficient conditions for the nonnegativity of the curvature operator, nonnegativity of the  sectional curvature, as well as the  nonnegativity of the  bisectional curvature respectively.
  In \cite{YZ}, Yang and Zheng later proved that the necessary and sufficient condition in \cite{WZ} for the nonnegativity of the  sectional curvature  holds for the nonnegativity of the complex sectional curvature under the unitary symmetry. In \cite{HT}, Huang and Tam  obtained the necessary and sufficient conditions for (NOB) and (NQOB) respectively. Moreover  they constructed a $\mathsf{U}(m)$-invariant K\"ahler metric on $\mathbb{C}^m$,  which is of  (NQOB), but does not have (NOB) nor nonnegativity of the Ricci curvature. In \cite{Ni-Niu2}, the construction was modified to illustrate an example with (NOB), but the holomorphic sectional curvature is negative somewhere. In later sections  we will construct  $\mathsf{U}(m)$-invariant K\"ahler metrics on $\mathbb{C}^m$ which has (NOB) but Ricci curvature is negative somewhere (this of course implies that holomorphic sectional curvature must be negative somewhere). We will also construct examples which has (NOB) and positive Ricci curvature, but the holomorphic sectional curvature is negative somewhere.

  We follow the same notations as in \cite{WZ, YZ}. 
  Let $(z_1, \cdots, z_m)$ be the standard coordinate on $\mathbb{C}^m$ and $r=|z|^2$. An $\mathsf{U}(m)$-invariant metric on $\mathbb{C}^m$ has the K\"ahler form
  \begin{equation}
  \omega=\frac{\sqrt{\!-\!1}}{2}\p\bar{\p} P(r)
  \end{equation}
  where $P\in C^\infty\left([0, +\infty)\right)$. Under the local coordinates, the metric has the components:
  \begin{equation}\label{eq:g}
  g_{i\bar{j}}=f(r)\delta_{ij}+f'(r)\bar{z}_i z_j.
  \end{equation}
  We further denote:
  \begin{equation}\label{eq:g1}
  f(r)=P'(r), \quad h(r)=(rf)'.
  \end{equation}
It is easy to check that $\omega$ will give a complete K\"ahler metric on $\mathbb{C}^n$ if and only if
  \begin{equation}\label{eq:10}
  f>0, \, h>0, \, \int_0^\infty \frac{\sqrt{h}}{\sqrt{r}}dr=+\infty.
  \end{equation}
  If $h>0$, then $\xi=-\frac{rh'}{h}$ is a smooth function on $[0, \infty)$ with $\xi(0)=0$. On the other hand, if $\xi$ is a smooth function on $[0, \infty)$ with $\xi(0)=0$, one can  define $h(r)=\exp(-\int_0^r \frac{\xi(s)}{s}ds)$ and $f(r)=\frac{1}{r}\int_0^r h(s)$ ds with $h(0)=1$. It is easy to see that $\xi(r)=-\frac{rh'}{h}$.
  Then (\ref{eq:g}) defines a $\mathsf{U}(m)$-invariant K\"ahler metric on $\mathbb{C}^m$.

   The components of the curvature operator of a $\mathsf{U}(m)$-invariant K\"ahler metric under the orthonormal frame $\{E_1=\frac{1}{\sqrt{h}}\p_{z_1}, E_2=\frac{1}{\sqrt{f}}\p_{z_2}, \cdots, E_m=\frac{1}{\sqrt{f}} \p_{z_m}\}$ at $(z_1, 0, \cdots, 0)$  are given as follows, see \cite{WZ}:
  \begin{eqnarray}
  	A&=& R_{1\bar{1}1\bar{1}}=-\frac{1}{h}\left(\frac{rh'}{h}\right)'=\frac{\xi'}{h}; \label{eq:A}\\
  	B&=& R_{1\bar{1}i\bar{i}}=\frac{f'}{f^2}-\frac{h'}{hf}=\frac{1}{(rf)^2}\left[rh-(1-\xi)\int_0^r h(s)\,  ds\right],\,  i\ge 2;\label{eq:B}\\
  	C&=& R_{i\bar{i}i\bar{i}}=2R_{i\bar{i}j\bar{j}}=-\frac{2f'}{f^2}=\frac{2}{(rf)^2}\left(\int_0^r h(s)\,  ds-rh\right),\, i\neq j, i, j\ge 2.\label{eq:C}
  \end{eqnarray}
   The other components of the curvature tensor are zero, except those obtained by the symmetric properties of curvature tensor.

   The following result was  proved in \cite{WZ}, which plays an important role in the construction.
   \begin{theorem}[Wu-Zheng]
   	(1) If $0<\xi<1$ on $(0, \infty)$, then $g$ is complete.
   	
   	(2) $g$ is complete and has positive  bisectional curvature if and only if $\xi'>0$ and $0<\xi<1$ on $(0, \infty)$, where $\xi'>0$ is equivalent to $A>0, B>0$ and $C>0$.
   	
   	(3) Every complete $\mathsf{U}(m)$-invariant K\"ahler metric on $\mathbb{C}^m$ with positive bisectional curvature is given by a smooth function $\xi$ in (2).
   \end{theorem}

It was proved in \cite{WZ}, \cite{HT}, and \cite{Ni-Niu2} that the following result holds.

\begin{proposition}\label{prop-51}
Let $g$ be a $\mathsf{U}(m)$-invariant K\"ahler metric on ${\mathbb C}^m$, with positive functions $f$, $h$ on $[0, \infty )$ described as above. Then
 \newline (i)\quad   $g$ has positive bisectional curvature $\iff$ $A>0$, $B>0$, $C>0$  $\iff$ $A>0$.
 \newline (ii) \quad If $m\geq 3$, then $g$ has positive orthogonal bisectional curvature $\iff$  $B>0$, $C>0$, $A+C>0$.
 \newline (iii)\quad  If $m\geq 3$, then $g$ has positive orthogonal bisectional and positive Ricci curvature $\iff$  $B>0$, $C>0$, $A+C>0$, $A+(m-1)B>0$.
\end{proposition}

Note that when $m=2$, the positivity of the orthogonal bisectional curvature no longer guarantees $C>0$, and the curvature condition for (ii) actually becomes $B>0$ and $A+C>0$; while the condition for (iii) becomes $B>0$, $A+B>0$, $C+B>0$, and $A+C>0$. In particular, the ``$\Longleftarrow $" part of (ii) and (iii) are still valid when $m =2$.

As noted in \cite{WZ}, there are plenty of metrics satisfying (i). In \cite{HT}, the authors perturbed metrics in (i) to obtain metrics in (ii) that are not in (i). For case (iii), as well as the comparison theorem proved earlier,    the following questions are natural (the first question was raised in (\cite{Ni-Niu2}):

\noindent {\bf Questions.} 1) Does there exist a complete $\mathsf{U}(m)$-invariant K\"ahler metric on ${\mathbb C}^m$ with positive orthogonal bisectional curvature, positive Ricci curvature, but does not have nonnegative holomorphic sectional curvature? Namely, a metric $g$ such that $B$, $C$, $A+C$, $A+(m-1)B$ are positive functions on $[0,\infty )$, while $A$ is negative somewhere.

2) Does there exist a complete $ \mathsf{U}(m)$-invariant K\"ahler metric on  ${\mathbb C}^m$ with positive orthogonal bisectional curvature and negative Ricci curvature somewhere?

\vspace{0.2cm}

In \cite{WZ}, the authors used the $\xi$ function to describe $\mathsf{U}(m)$-invariant K\"ahler metrics on ${\mathbb C}^m$, which is defined by
$ \xi = - \frac{rh'}{h}$.
Clearly, $\xi$ is smooth on $[0,\infty )$ with $\xi (0)=0$, and is determined by $g$. Conversely, $\xi$ determines $h$ and $f$ up to a positive constant multiple, and as proved in \cite{WZ}, if $0<\xi <1$ in $(0,\infty )$, then the metric $g$ determined by $\xi$ is complete.

In terms of $\xi$, the above question (i) can be rephrased (see the last paragraph of \cite{Ni-Niu2}) as finding a smooth function $\xi$ on $[0,\infty )$ with $\xi(0)=0$ and $0< \xi <1$ on $(0,\infty )$, such that $\xi'<0$ somewhere, yet
\begin{eqnarray*}
& & rh - (1-\xi )\int_0^r \! h(s)ds  > 0; \\
& & \int_0^r \! h(s)ds  - rh  \ > 0;\\
& & \xi' + \frac{2h}{(rf)^2} \ \big( \int_0^r\!h(s)ds  - rh \big) \ > 0; \\
& & \xi' + \frac{(m-1)h}{(rf)^2} \ \big( rh - (1-\xi )\int_0^r \! h(s)ds \big) \ > 0
\end{eqnarray*}
everywhere on $(0,\infty )$.

It is not obvious why such a function must exist. So we will resort to another characterization of $\mathsf{U}(m)$-invariant metrics in \S 5 of \cite{WZ}  by the generating surface of revolution.

\section{Examples--a characterization}

Let us first recall the characterization of $\mathsf{U}(m)$-invariant metrics by surface of revolutions given in \S 5 of \cite{WZ}. Let $g$ be a complete $\mathsf{U}(m)$-invariant K\"ahler metric on ${\mathbb C}^m$, with $h$, $f$ defined as before. Let us assume that $h'<0$ everywhere. Write $\xi = - \frac{rh'}{h}$, then we have $0<\xi < 1$ on $(0,\infty )$ by the assumption $h'<0$ and the completeness of $g$.

Write $x=\sqrt{rh}$. On $(0,\infty )$, we have $x'= \frac{\sqrt{h} (1-\xi )} {2\sqrt{r}} >0$, so $x$ is a strictly increasing function and $x'^2< \frac{h}{4r}$. Define a positive, strictly increasing function $y$ on $(0,\infty )$ so that $y(0^+)=0$ and
$$ x'^2 + y'^2 = \frac{h}{4r}.$$
The metric $g$ is determined by the smooth function $y=F(x)$ on $(0,x_0)$, where $x_0=\lim_{r\rightarrow \infty }\sqrt{rh}\leq \infty $. It is easy to see that $F$ is actually smooth on $[0,x_0)$ and $F(0)=0$. From the definition, we have the relationship
$$ 1+ \big(\frac{dF}{dx}\big)^2 = \frac{1}{(1-\xi )^2}.$$
As computed in \cite{WZ}, in terms of this generating function $F(x)$, the curvature component functions are
$$ A = \frac{F'F''}{2x(1+F'^2)^2} , \ \ \ \  B = \frac{1}{v^2} \big( x^2\!- \!\frac{v}{\sqrt{1+F'^2}} \big) , \ \ \ \ C = \frac{2}{v^2} (v-x^2), $$
where
$$ v (x)= rf = \int_0^x 2\tau \sqrt{1+F'^2(\tau )} d\tau.$$
To simplify these expressions, let us use the trick in \cite{WZ} by letting
$$ F(x) = \frac{1}{2}p(x^2), \ \ \ \  p(t)=\int_0^t\! \sqrt{q(\tau )}\ d\tau , \ \ \ \  q(t) = \frac{(k(t))^2-1} {t} $$
where $k(t)$ is a smooth function on $[0,\infty )$ such that $k(0)=1$ and $k(t)>1$ when $t>0$. We have
$$ F'(x) = xp'(x^2)=x \sqrt{q(x^2)},$$
therefore
$$ 1+F'^2(x) = 1+x^2q(x^2)=(k(x^2))^2. $$
Now let us denote by $t=x^2$, and $u(t)=\int_0^tk(\sigma )d\sigma $, then by a straight forward computation, we get
$$ A=\frac{k'}{k^3}, \ \ \ \ B = \frac{1}{ku^2}(tk-u), \ \ \ \ C=\frac{2}{u^2}(u-t). $$
Write $u(t)=t+t\alpha (t)$. Then $k=u'=1+\alpha +t\alpha'$, and
\begin{equation}\label{eq:ch1} A=\frac{t\alpha''+2\alpha'}{(1+\alpha +t\alpha')^3}, \ \  B = \frac{\alpha'}{(1+\alpha +t\alpha')(1+\alpha)^2}, \ \  C=\frac{2\alpha }{t(1+\alpha )^2}.
\end{equation}

\section{Examples with (NOB), positive Ricci, but negative holomoprhic sectional curvature somewhere}

The goal here is to prove the following result, which affirmatively answers a question in \cite{Ni-Niu2}.

\begin{theorem}\label{thm-example1}
For any $m\geq 2$, there are complete $\mathsf{U}(m)$-invariant K\"ahler metrics on  ${\mathbb C}^m$ with positive Ricci curvature and positive orthogonal bisectional curvature everywhere, yet the holomorphic sectional curvature is negative somewhere.
\end{theorem}

Now that the expressions of the curvature components are reasonably simple, we could try to find functions $\alpha$ so that the desired curvature conditions are satisfied. For instance, let us consider the smooth function $\alpha (t)$ given by
\begin{equation}\label{eq-71}
\alpha (t) = \lambda \big(1 - \frac{1}{(1+t^2)^a}\big) ,
\end{equation}
where $a$, $\lambda$ are positive constants with $a\in (\frac{1}{2},1)$. We have $\alpha (0)=0$, and
$$\alpha '=\frac{2a\lambda t}{(1+t^2)^{a+1}}.$$
So $\alpha $ and $\alpha'$ are positive on $(0,\infty )$.  Note that the function $\alpha'$ and $\frac{\alpha }{t}$ are actually also positive at $t=0$. By formula (6.1) in the previous section, we have $B>0$, $C>0$ everywhere. Note that $A(0)>0$ as well, so the bisectional curvature of the metric $g$ is positive at the origin.

Let us examine the situation away from the origin. For a constant $b>0$, we compute
$$ (t^b \alpha')' = \frac{2a\lambda t^b}{(1+t^2)^{a+2}} \big( (b+1)+(b-1-2a)t^2\big). $$
For $b=2$, the right hand side factor becomes $3-(2a-1)t^2$, so the sign of $A$, or equivalently the sign of $t\alpha''+2\alpha'$, is the same as that of $(t_0-t)$, where $t_0=\sqrt{\frac{3}{2a-1}}$. That is, we have
$$ A>0  \ \ \mbox{on} \ [0, t_0) ,\ \ \ \mbox{and} \ \ A<0 \  \ \mbox{on}  \  (t_0, \infty ). $$
For $b=3$, $b-1-2a=2-2a>0$, so $(t^3\alpha')'>0$, thus by formula (6.1)
$$ k^3(A+(n-1)B) \geq k^3(A+B) \geq t\alpha'' + 3\alpha' > 0.$$
It remains to check the condition $A+C>0$. We have
$$k^3C\geq (1+\alpha +t\alpha' )\frac{2\alpha}{t} .$$
So when $2\alpha \geq 1$, we have $k^3C\geq \alpha '$, hence $k^3(A+C) \geq  t\alpha '' +3\alpha ' >0$.
Let us fix $a\in (\frac{1}{2}, 1)$, and choose $\lambda $ sufficiently large so that
$$ \frac{1}{2\lambda -1} < \left(1+\frac{3}{2a-1}\right)^a - 1,$$
in this case we have
$$ \left( \frac{2\lambda }{2\lambda -1} \right)^{\frac{1}{a}} -1 < \frac{3}{2a-1}. $$
Note that
$$ 2\alpha <1 \iff 1-\frac{1}{(1+t^2))^a} <\frac{1}{2\lambda} \iff t^2 < \left(\frac{2\lambda }{2\lambda -1}\right)^{\frac{1}{a}} -1 .$$
So by our choice of $\lambda$ we have $t< t_0=\sqrt{\frac{3}{2a-1}}$. But in this case $A>0$, thus $A+C>0$ as well.

This completes the proof of Theorem \ref{thm-example1}. Note that the metric $g$ given by $\alpha$ in (2) has positive bisectional curvature in a ball $B_c$, while outside the ball, at every point the holomorphic sectional curvature is negative in some direction.

One can also construct examples satisfying Theorem \ref{thm-example1} while the bisectional curvature is positive outside an annulus, in particular, outside a compact subset. To see such an example, let us consider
\begin{equation}
\alpha = t-2at^2+t^3,
\end{equation}
where $a>0$ is a constant to be determined. We have
\begin{eqnarray*}
&& \frac{\alpha }{t} = 1 - 2at + t^2 \\
& & \alpha ' = 1 - 4at +3t^2\\
& &  t\alpha'' +2\alpha '= 2(1-6at+6t^2)\\
& &  t\alpha'' +3\alpha '=3-16at+15t^2\\
&& t\alpha'' +2\alpha '+\frac{2\alpha}{t} = 2(2-8at+7t^2)
\end{eqnarray*}
We want to choose $a$ so that the middle line is negative somewhere, while the other four are positive everywhere in $(0,\infty )$. The first two guarantee that $B>0$, $C>0$, while last two imply that $A+B>0$, $A+C>0$. The middle term shares the same sign with $A$.

Note that for positive constants $a$, $b$, $c$, the polynomial $a-bt+ct^2$ will be everywhere positive on $[0,\infty )$ if and only if $b^2 < 4ac$, and when $b^2>4ac$, the polynomial will be negative in the interval $[t_1, t_2]$ where $t_1>t_2>0$ are the two roots. Applying this criteria to the five quadratic polynomials above, we know that we want respectively
$$ a^2<1, \ \ \ a^2<\frac{3}{4}, \ \  \ a^2> \frac{2}{3}, \ \ \ a^2 < \frac{45}{64}, \ \ \  a^2<\frac{7}{8} .$$
Since $\frac{2}{3} < \frac{45}{64} < \frac{3}{4}$, if we choose $a>0$ so that $\frac{2}{3}<a^2<\frac{45}{64}$, then the corresponding metric $g$ will have positive orthogonal bisectional and positive Ricci curvature everywhere, while the holomorphic sectional curvature is negative in some directions at every point in an annulus. The bisectional curvature is positive outside the annulus.

\section{The examples with (NOB) but negative Ricci curvature and negative holomorphic sectional curvature  somewhere}

We present here two constructions. The first one is along the line of \cite{HT} (see also \cite{Ni-Niu2}).
 Let $\xi$ be a smooth function on $[0, \infty)$ with $\xi(0)=0, \xi'(r)>0$ and $0<\xi(r)<1$ for $0<r<\infty$. Let $a=\lim_{r\rightarrow \infty} \xi(r)$. Then $0<a\le 1$. By the the discussion in the pervious sections,  this gives a complete $\mathsf{U}(m)$-invariant metric on $\mathbb{C}^m$ with positive bisectional curvature. The strategy of \cite{HT} is to perturb this metric by adding a perturbation term to $\xi$. This then  yields  one metric with the needed property.  It starts with some estimates for the metrics with positive bisectional curvature. In \cite{HT, WZ}   the following estimates (cf. Lemma 4.1 of \cite{HT}) were obtained.
 \begin{lemma}\label{lm:HT}
 	Let $\xi$ be as above with $\lim_{r\to \infty}\xi =a \, (\in (0, 1))$. The following holds\\
 	(1) $\lim_{r\to \infty} h(r)=0$ and $\lim_{r\to\infty} \frac{h(r+r_0)}{h(r)}=1$ for any $r_0>0$.\\
(2) For any $r>0$, $\left( r h-(1-\xi)\int_0^r h\right)' >0$, and
 $$
 \lim_{r\to \infty} \int_0^r h=\infty, \quad  \lim_{r\to \infty} h=0, \quad \lim_{r\to \infty} \frac{rh}{\int_0^r h}=1-a.
 $$
(3) For any $\epsilon>0$, and for any $r_0>0$, there is $R>r_0$ such that
 	$$
 	\xi'(R)-\epsilon h(R)C(R)<0.
 	$$
 	(4) $\lim_{r\rightarrow \infty} h(r)C(r)=0$.\\
 	(5) For all $\epsilon>0$, there exists $\delta>0$ such that if $R\ge 3$, $\delta\ge \eta\ge 0$ is a smooth function with support in $[R-1, R+1]$, then for all $r\ge 0$,
 	$$
 	h(r)\le \bar{h}(r)\le (1+\epsilon)h(r), \quad\mbox{and}\, \int_0^r h\le \int_0^r \bar{h}\le (1+\epsilon)\int_0^r h,
 	$$
 	where $\bar{h}(r)=\exp(-\int_0^r \frac{\bar{\xi}}{t}dt)$ and $\bar{\xi}=\xi-\eta$.
 \end{lemma}

Let $\phi$ be a cutoff function on $\mathbb{R}$ as in \cite{HT} such that

(i) $0\le \phi\le c_0$ with $c_0$ being an absolute constant;

(ii) $\mbox{supp} (\phi) \subset [-1, 1]$;

(iii) $\phi'(0)=1$ and $|\phi'|\le 1$.

The construction is to perturb $\xi$ into $\bar{\xi}(r)=\xi(r)-\alpha h(R)C(R)\phi(r-R)$ for suitable choice of $R$, $\alpha$. Note that this only changes the value of $\xi$ on a compact set. Once $\bar{h}$ is defined, equations (\ref{eq:A})--(\ref{eq:C}) define the corresponding curvature
components $\bar{A}, \bar{B}, \bar{C}$ of the perturbed metric.
\begin{theorem}\label{thm-example}
	There is $1>\alpha>0$ such that for any $r_0>0$ there is $R>r_0$ satisfying the following: If $\bar{\xi}(r)=\xi(r)-\alpha h(R)C(R)\phi(r-R)$, then $\bar{\xi}$ determines a complete $\mathsf{U}(m)$-invariant K\"ahler metric on $\mathbb{C}^m$ such that	
	\begin{enumerate}
		\item $\bar{A}+\bar{C}>0$ on $[R-1, R+1]$;
		\item $\bar{B}(r)>0$ for all $r$;
		\item $\bar{C}(r)>0$ for all $r$; and
\item $\bar{A}(R)+(m-1)\bar{B}(R)<0$.
			\end{enumerate}
		Then $\bar{\xi}$ will give a compete $\mathsf{U}(m)$-invariant K\"ahler metric which satisfies (NOB) but does not have nonnegative Ricci curvature, nor nonnegative holomorphic sectional curvature.
\end{theorem}
   \begin{proof} Note that (1)-(3) implies the (NOB). The estimate (4) shows the negativity of the Ricci somewhere. First  for any $\alpha>0$, by choosing $R$ large, $\bar{\xi}$ (along with the $\bar{h}$ and $\bar{f}$) defines a complete K\"ahler metric on $\mathbb{C}^m$. Recall that $a\in (0, 1)$ is the limit of $\lim_{r\to \infty}\xi(r)$, $c_0$ being the bound of $|\phi|.$  The proof of (2) and (3) is exactly the same as in \cite{Ni-Niu2}, which does not involve the careful picking of $\alpha>0$. We need to choose the constant $\alpha>0$ a bit more carefully here to achieve both (1) and (4) simultaneously. Note that in \cite{HT}, metrics were constructed with both $\bar{A}(R)+\bar{C}(R)$ and $\bar{A}(R)+(m-1)\bar{B}(R)$ being negative.

   By (\ref{eq:A}) and (\ref{eq:C}) for (1) we only need to prove if for $r\in [R-1, R+1]$.
   By the formula (\ref{eq:C}) and the proof of Lemma 4.2 in \cite{HT} (precisely (4.6) of \cite{HT}), we may choose a large $r_1$ so that if $R>r_1$ and for $r\in [R-1, R+1]$,
	$$
	\bar{C}(r)\ge \frac{2}{(1+\epsilon)^2\int_0^R h}(a-2\epsilon+a\epsilon-\epsilon^2)
	$$
	provided $a-2\epsilon+a\epsilon-\epsilon^2>0$. For $\epsilon>0$ sufficiently small it clearly satisfies this condition. Here $a>$ is the constant from Lemma \ref{lm:HT}. On the other hand,
	$$
	C(R)\le \frac{2}{\int_0^R h}(a+\epsilon)
	$$
	if $r_1$ is large enough depending only on $\epsilon$ and $R>r_1$. Hence, if $\epsilon$ and $r_1$ satisfy the above conditions, then for $r\in [R-1, R+1]$,
	$$
	\bar{C}(r)\ge \frac{a-2\epsilon+a\epsilon-\epsilon^2}{(a+\epsilon)(1+\epsilon)^2}C(R).
	$$
	Therefore, if $\epsilon>0$ satisfies $a>\epsilon$ and $a-2\epsilon+a\epsilon-\epsilon^2>0$,  we can find $r_1>r_0$ such that if $R>r_1$, then it holds for $r\in  [R-1, R+1]$,
	\begin{eqnarray}
	\bar{A}(r)+\bar{C}(r)&\ge& \frac{\xi'(r)-\beta}{\bar{h}}+\bar{C}(r)\nonumber\\
	&\ge& \frac{-\beta}{\bar{h}(r)}+ \frac{a-2\epsilon+a\epsilon-\epsilon^2}{(a+\epsilon)(1+\epsilon)^2}C(R)\\
	&\ge & -\frac{\beta}{(1-\epsilon)h(R)}+\frac{a-2\epsilon+a\epsilon-\epsilon^2}{(a+\epsilon)(1+\epsilon)^2}C(R)\nonumber\\
	&= &\frac{1}{(1-\epsilon)h(R)}[-\beta+(1-\epsilon)\frac{a-2\epsilon+a\epsilon-\epsilon^2}{(a+\epsilon)(1+\epsilon)^2}h(R)C(R)]\nonumber.
	\end{eqnarray}
In the third line we have used the fact that $-\frac{\beta}{\bar{h}(r)}\ge -\frac{\beta}{h(r)}\ge -\frac{\beta}{h(R+1)}$ and $\lim_{r\to \infty} \frac{h(r)}{h(r+r_0)}=1$. Hence for $r\in  [R-1, R+1]$,
$$
\bar{A}(r)+\bar{C}(r)\ge \frac{1}{(1-\epsilon)h(R)}[-\alpha +(1-\epsilon)\frac{a-2\epsilon+a\epsilon-\epsilon^2}{(a+\epsilon)(1+\epsilon)^2}]h(R)C(R).
$$
Hence if we pick $\alpha=\frac{1}{2}$, for sufficiently small $\epsilon$ we can be sure that $\bar{A}(r)+\bar{C}(r)>0$.  This proves (1).

On the other hand, as in \cite{HT}, for $r_1\ge r_0$ sufficiently large and $R\ge r_1$,
\begin{eqnarray*}
\bar{C}(r)&=&\frac{2}{\int_0^r \bar{h}}\left(1-\frac{r\bar{h}}{\int_0^r \bar{h}}\right)\\
&\le& \frac{2}{\int_0^r \bar{h}}\left(1-\frac{r h}{(1+\epsilon)\int_0^r h}\right)\\
&\le& \frac{2}{\int_0^r h}\frac{a+2\epsilon}{1+\epsilon}.
\end{eqnarray*}
Here we have used part (ii) of Lemma \ref{lm:HT}. But
$$
C(r)= \frac{2}{\int_0^r h}\left(1-\frac{rh}{\int_0^r h}\right)\ge \frac{2}{\int_0^r h}(a-\epsilon).$$ Hence for small $\epsilon$
$$
\bar{C}(r)\le \frac{a+2\epsilon}{(a-\epsilon)(1+\epsilon)} C(r)
$$
This implies that
\begin{eqnarray*}
\bar{A}(R)+\frac{1}{3}\bar{C}(R)&=&\frac{\xi'(R)-\alpha C(R)h(R)}{\bar{h}(R)}+\frac13\bar{C}(R)\\
&\le&\frac{\xi'(R)-\alpha C(R)h(R)}{(1+\epsilon)h(R)}+\frac13 \frac{a+2\epsilon}{(a-\epsilon)(1+\epsilon)} C(R)\\
&=&\frac{1}{(1+\epsilon) h(R)}\left(\xi'(R)-\alpha C(R)h(R)+\frac13 \frac{a+2\epsilon}{a-\epsilon} C(R)h(R)\right).
\end{eqnarray*}
Noting part (iii) of Lemma \ref{lm:HT}, and that we picked $\alpha=\frac{1}{2}$, for sufficiently small $\epsilon$ we have that
\begin{equation}\label{eq:ric-hp1}
\bar{A}(R)+\frac{1}{3}\bar{C}(R)\le 0.
\end{equation}
On the other hand, similar calculation as the above shows that
\begin{eqnarray*}
\bar{C}(r)&\ge& \frac{2}{(1+\epsilon)\int_0^r h}\left(1-\frac{(1+\epsilon) rh}{\int_0^r h}\right)\\
&\ge& \frac{2}{(1+\epsilon)\int_0^r h} \left(1-(1+\epsilon)(1-a+\epsilon)\right)\\
&\ge&\frac{2(a-\epsilon)}{\int_0^r h}.
\end{eqnarray*}
Here $\epsilon$ is small and we may choose a different one in the last line. Thus together with (\ref{eq:ric-hp1}) we have
$$
\bar{A}(R)\le -\frac13 \frac{2(a-\epsilon)}{\int_0^r h}.
$$
On the other hand, as in Lemma 4.2 of \cite{HT}, for $R$ sufficiently large,
$$
\bar{B}(R)\le \frac{\epsilon}{\int_0^R h}.
$$
Combining them we conclude that $\bar{A}(R)+(m-1)\bar{B}(R)<0$ for $R\ge r_1$. This proves (4).
   \end{proof}

One could also construct $\mathsf{U}(m)$-invariant complete K\"ahler metrics on ${\mathbb C}^m$ with $B^{\perp }>0$ but $Ric\ngeq 0$ and $H\ngeq 0$, using the notations and the construction in the previous section. Below are the details.

For the sake of simplicity, we will work with the $m=2$ case. In this case, $B^{\perp}$ coincides with $Ric^{\perp}$, and by Proposition \ref{prop-51} and the remark afterwards, its positivity means $B>0$ and $A+C>0$. So to ensure that the Ricci and the holomorphic sectional curvature $H$ are not everywhere nonnegative, we need $A\ngeq 0$ and $A+B\ngeq 0$. That is, it suffices to find such a metric satisfying
$$B>0, \ \ C>0, \ \ A+C>0, \ \ A\ngeq 0, \ \ A+B\ngeq 0,$$
where the functions $A$, $B$, $C$ are expressed in terms of the $\alpha $ function by formulae in  $(\ref{eq:ch1})$.  As in the previous sections, we may start with the function
\begin{equation*}
\alpha (t) = \lambda \big(1 - \frac{1}{(1+t^2)^a}\big) ,
\end{equation*}
where $a$, $\lambda$ are positive constants with $a> \frac{1}{2}$. We will specify the range of $a$ and $\lambda $ later. As before, we have $\alpha (0)=0$, and
$$\alpha '=\frac{2a\lambda t}{(1+t^2)^{a+1}}.$$
 So $\alpha $ and $\alpha'$ are positive on $(0,\infty )$.  Also, the function $\alpha'$ and $\frac{\alpha }{t}$ are positive at $t=0$, so we have $B>0$, $C>0$ everywhere, while $A$ has the same sign with $t_0-t$, where $t_0=\sqrt{\frac{3}{2a-1}}$. In particular, $A\ngeq 0$.

As noticed before, for any $b>0$, we have
$$ (t^b \alpha')' = \frac{2a\lambda t^b}{(1+t^2)^{a+2}} \big( (b+1)+(b-1-2a)t^2\big). $$
This function will be positive on $(0,\infty )$ if $b\geq 1+2a$, and negative for large $t$ if $b<1+2a$.

In the following, we will take $a=6$. So $t_0=\sqrt{\frac{3}{11}}$. Clearly, we can choose $\lambda >0$ large enough so that $\alpha(t_0)>6$. Since $\alpha$ is strictly increasing, when $\alpha<6$, we must have $t<t_0$, thus $A>0$ hence $A+C>0$. While when $\alpha \geq 6$, we have
\begin{eqnarray*}
(A+C) (1+\alpha +t\alpha')^3 & = & t\alpha'' +2\alpha ' +  \frac{2\alpha }{t} (1+\alpha +t\alpha') \frac{(1+\alpha +t\alpha')^2}{(1+\alpha )^2} \\
& \geq & t\alpha'' +2\alpha ' +  \frac{2\alpha }{t} (1+\alpha +t\alpha') \\
& \geq & t\alpha'' +2\alpha ' +  12 \alpha' \ = \  t^{-13} (t^{14}\alpha')' \ > \ 0
\end{eqnarray*}
since $14>2a+1=13$. This demonstrates that $A+C>0$ everywhere. To see that $A+B\ngeq 0$, let us observe that the inequality
$$ t\alpha' \leq 2\alpha $$
always holds, as it is implied by
$$ \frac{at^2}{(1+t^2)^{a+1}} \leq 1 - \frac{1}{(1+t^2)^{a}}, $$
which is true since $1+at^2 \leq (1+t^2)^a$ for any $t$. So we now have
$$1+\alpha +t\alpha' \leq 3(1+\alpha).$$
Thus the quantity $(A+B) (1+\alpha +t\alpha')^3 $ can be estimated as
\begin{eqnarray*}
(A+B) (1+\alpha +t\alpha')^3 & = & t\alpha'' +2\alpha' + \big(\frac{1+\alpha +t\alpha'}{1+\alpha }\big)^2 \alpha' \\
& \leq & t\alpha'' +2\alpha' + 3^2\alpha' \\
& = &  t^{-10} (t^{11}\alpha')'.
\end{eqnarray*}
Since $11<2a+1=13$, $(t^{11}\alpha')' < 0$ when $t$ is large. So we have $A+B\ngeq 0$ as desired. This construction gives the metric which has $B^{\perp}>0$, but has negative holomorphic sectional curvature and negative Ricci curvature outside a compact subset.

\section{Appendix}

In this appendix, we will give the calculation of the curvature for the surface $M^2= {\mathbb P}^2 \# \overline{{\mathbb P}^2 }$, when the metric is the restriction of the product metric. Consider
$$ M^2 = \{ ([u_0\!:\!u_1\!:\!u_2], \ [v_1\!:\!v_2]) \in {\mathbb P}^2  \times {\mathbb P}^1 \mid u_1v_2 = u_2v_1 \}, $$
 and let $\omega_g$ be the restriction on $M$ of the product metric
 $$ \omega_g = \sqrt{\!-\!1} \partial \overline{\partial} \log (|u_0|^2 \!+\!|u_1|^2 \!+\! |u_2|^2) +   \lambda \sqrt{\!-\!1} \partial \overline{\partial} \log (|v_1|^2 \!+\!|v_2|^2 ) $$
where $\lambda >0$ is a constant. We will prove the following

\begin{proposition}
The surface $(M^2,\omega_g)$ will have its Ricci curvature positive everywhere if and only if $\lambda >\frac{1}{2}$, and it will have its holomorphic sectional curvature positive everywhere if and only if $\lambda >1$.
\end{proposition}

To see this, let us fix an arbitrary point $p\in M$. First let us consider the case when $u_0(p)\neq 0$. By a unitary change of coordinate in $(u_1,u_2)$ and $(v_1,v_2)$, we may assume that $p=([1\!:\!a\!:\!0], \ [1\!:\!0])$, where $a\in [0,\infty )$. So in a neighborhood of $p$, we have local holomorphic coordinate $(z_1,z_2)$ which corresponds to the point $([1\!:\!z_1\!:\! z_1z_2],  [1\!:\!z_2])$, and $p=(a,0)$. In this neighborhood, the metric $\omega_g$ becomes
$$ \omega_g = \sqrt{\!-\!1}\partial \overline{\partial} \log  \eta +  \lambda  \sqrt{\!-\!1}\partial \overline{\partial} \log \sigma $$
where $\sigma = 1+|z_2|^2$ and $\eta = 1+|z_1|^2\sigma = 1+ |z_1|^2 + |z_1z_2|^2$. We compute that
$$ g_{1\overline{1}} = \frac{\sigma }{\eta^2}, \ \ g_{1\overline{2}} = \frac{\overline{z}_1z_2}{\eta^2}, \ \  g_{2\overline{2}} = \frac{|z_1|^2(|z_1|^2\!+\!1)}{\eta^2} + \frac{\lambda }{\sigma^2} .$$
From this, we get
\begin{eqnarray*}
g_{1\overline{1},1} & = & -\frac{2}{\eta^3} \sigma^2 \overline{z}_1, \ \ \  \ \ \ \ g_{2\overline{2},2} \ = \  -\frac{2}{\eta^3}|z_1|^4(|z_1|^2\!+\!1)\overline{z}_2- \lambda\frac{2}{\sigma^3} \overline{z}_2,  \\
g_{1\overline{2},1} & = & -\frac{2}{\eta^3} \sigma \overline{z}_1^2z_2, \ \ \  \ \ \ \ \ \ \ \ g_{1\overline{2},2} \ = \  \frac{1}{\eta^2} \overline{z}_1 -\frac{2}{\eta^3}|z_1|^2|z_2|^2\overline{z}_1,  \\
g_{1\overline{2},\overline{1}} & = & \frac{1}{\eta^2} z_2 - \frac{2}{\eta^3} \sigma |z_1|^2 z_2, \ \ \ \ \  \ \ \ \ g_{1\overline{2},\overline{2}} \ = \  -\frac{2}{\eta^3}|z_1|^2\overline{z}_1z_2^2.
\end{eqnarray*}
Under the local coordinate $(z_1,z_2)$, the curvature components are given by
$$ R_{i\overline{j}k\overline{\ell}} = - g_{i\overline{j},k\overline{\ell}} + \sum_{p,q=1}^2 g_{i\overline{p},k} \ \overline{g_{j\overline{q},\ell }} \ g^{\overline{p}q}.$$
At the point $p=(a,0)$, we have $\eta=1+a^2$, $\sigma=1$, and
\begin{eqnarray*}
 && g_{1\overline{1}}=\frac{1}{\eta^2},  \ \ \  g_{1\overline{2}}=0,  \ \ \ g_{2\overline{2}}=\frac{a^2\!+\!\lambda \eta}{\eta}, \\
&&  g_{1\overline{2},1} = g_{2\overline{1},1} = g_{2\overline{1},2} = g_{1\overline{1},2} = g_{2\overline{2},2}=0.
\end{eqnarray*}
From these, we get that at $p$
$$ R_{1\overline{2}1\overline{2}} = - g_{1\overline{2},1\overline{2}} + \frac{1}{g_{1\overline{1}}} g_{1\overline{1},1} \ \overline{g_{2\overline{1},2}} + \frac{1}{g_{2\overline{2}}} g_{1\overline{2},1} \ \overline{g_{2\overline{2},2}} \ = \ 0 .$$
Similarly, we also get $ R_{1\overline{1}1\overline{2}}=  R_{1\overline{2}2\overline{2}} =0$ at $p$. Next, we compute at $p$
\begin{eqnarray*}
R_{1\overline{1}2\overline{2}} & = & - g_{1\overline{2},2\overline{1}} + \frac{1}{ g_{1\overline{1}}  } |g_{1\overline{1},2}|^2 + \frac{1}{ g_{2\overline{2}}  } |g_{1\overline{2},2}|^2 \\
& = &  -\frac{1}{\eta^2} + \frac{2a^2}{\eta^3} + 0 + \frac{\eta}{a^2+\lambda \eta } \ \frac{a^2}{\eta^4} \ = \  \frac{a^2-1}{\eta^3} + \frac{a^2}{\eta^3(a^2+\lambda \eta )}  , \\
R_{1\overline{1}1\overline{1}} & = & - g_{1\overline{1},1\overline{1}} + \frac{1}{ g_{1\overline{1}}  } |g_{1\overline{1},1}|^2 + \frac{1}{ g_{2\overline{2}}  } |g_{1\overline{2},1}|^2 \\
& = & \frac{2}{\eta^3} -\frac{6a^2}{\eta^4} + \eta^2 \ \frac{4a^2}{\eta^6} + 0 \ = \  \frac{2}{\eta^4}, \\
R_{2\overline{2}2\overline{2}} & = & - g_{2\overline{2},2\overline{2}} + \frac{1}{ g_{1\overline{1}}  } |g_{2\overline{1},2}|^2 + \frac{1}{ g_{2\overline{2}}  } |g_{2\overline{2},2}|^2 \\
& = & - g_{2\overline{2},2\overline{2}}  \ = \ \frac{2a^4}{\eta^2} + 2\lambda.
\end{eqnarray*}
Now let us compute the component of the Ricci curvature at $p$. We have $R_{1\overline{2}}=0$, and
\begin{eqnarray*}
R_{1\overline{1}} & = & \eta^2 R_{1\overline{1}1\overline{1}} + \frac{\eta}{a^2\!+\!\lambda\eta} R_{1\overline{1}2\overline{2}} \\
& = & \frac{2}{\eta^2} + \frac{a^2-1}{\eta^2(a^2\!+\!\lambda \eta ) } + \frac{a^2}{\eta^2(a^2\!+\!\lambda \eta )^2 },\\
R_{2\overline{2}} & = & \eta^2 R_{1\overline{1}2\overline{2}} + \frac{\eta}{a^2\!+\!\lambda \eta} R_{2\overline{2}2\overline{2}} \\
& = & \frac{a^2-1}{\eta} + \frac{a^2+2a^4+2\lambda \eta^2 }{\eta (a^2+\lambda \eta )}.
\end{eqnarray*}
Since $2\lambda \eta^2 > \lambda \eta$, we know that $R_{2\overline{2}} >0$ for all $a\geq 0$. For $R_{1\overline{1}}$, if we let $f(t)$ be the function of $t=a^2$ which represents the quantity $\eta^2(a^2+\lambda \eta )^2R_{1\overline{1}}$, then
$$ f(t) = (\lambda +1)(2\lambda +3)t^2 + 4\lambda (\lambda +1)t +\lambda (2\lambda -1). $$
Hence  $R_{1\overline{1}}>0$ for all $a\geq 0$ if and only if $\lambda > \frac{1}{2}$.

Next let us examine the holomorphic sectional curvature $H$ at the point $p$. For any tangent direction $X=x_1 \frac{\partial}{\partial z_1}+x_2\frac{\partial}{\partial z_2}$ at $p$, we have
\begin{eqnarray}
 R_{X\overline{X}X\overline{X}} & =  & |x_1|^4R_{1\overline{1}1\overline{1}} + |x_2|^4R_{2\overline{2}2\overline{2}} +4|x_1x_2|^2R_{1\overline{1}2\overline{2}}\label{eq:a1}\\
 & = & \frac{2}{\eta^4}|x_1|^4 + \frac{2}{\eta^2} (a^4\!+\!\lambda \eta^2) |x_2|^4 + \frac{4}{\eta^3} \big(a^2\!-\!1\!+\!\frac{a^2}{a^2\!+\!\lambda \eta }  \big)  |x_1x_2|^2. \nonumber
\end{eqnarray}
In particular, when $a=0$, we have
$$ R_{X\overline{X}X\overline{X}} = 2|x_1|^4+2\lambda |x_2|^4 - 4|x_1x_2|^2, $$
so if $\lambda < 1$, then there are $X\neq 0$ with $R_{X\overline{X}X\overline{X}}<0$, while when $\lambda =1$, we have $R_{X\overline{X}X\overline{X}}\geq 0$ but attains $0$. Now suppose that $\lambda >1$. If $x_2=0$, then $R_{X\overline{X}X\overline{X}}> 0$. If $x_2\neq 0$, then by (\ref{eq:a1}) we have
\begin{eqnarray*}
R_{X\overline{X}X\overline{X}} &\geq & \frac{2}{\eta^4} |x_1|^4 + 2\lambda |x_2|^4 - \frac{4}{\eta^3} |x_1x_2|^2 \\
& > & \frac{2}{\eta^4} |x_1|^4 + 2 |x_2|^4 - \frac{4}{\eta^3} |x_1x_2|^2 \\
&\geq & 2\sqrt{ \frac{2}{\eta^4}|x_1|^4 \ 2|x_2|^4} - \frac{4}{\eta^3} |x_1x_2|^2\\
& = & \frac{4}{\eta^2} |x_1x_2|^2 - \frac{4}{\eta^3} |x_1x_2|^2 \ = \  \frac{4a^2}{\eta^3} |x_1x_2|^2 \ \geq \ 0.
\end{eqnarray*}
So when $\lambda >1$, the holomorphic sectional curvature at $p$ is positive.

Now let us assume that $u_0(p)=0$, namely, $p$ lies in the line at infinity with respect to the point of blowing up. Again by a simultaneous unitary coordinate change on the $(u_1,u_2)$ and $(v_1,v_2)$ if necessary, we may assume that $p=([0\!:\!1\!:\!0], [1\!:\!0])$. Let us choose holomorphic coordinate $(z_1,z_2)$ near $p$ by letting it correspond to the point $(z_1\!:\!1\!:\!z_2], [1\!:\!z_2])$. Then $p=(0,0)$, and the metric in this case is given by
$$ \omega_g = \sqrt{\!-\!1}\partial \overline{\partial} \log  (1\!+\!|z_1|^2\!+\!|z_2|^2) +  \lambda  \sqrt{\!-\!1}\partial \overline{\partial} \log (1\!+\!|z_2|^2). $$
Again if we denote by $\eta = 1\!+\!|z_1|^2\!+\!|z_2|^2$ and $\sigma = 1\!+\!|z_2|^2$, then we have
\begin{eqnarray*}
 g_{i\overline{j}} & = & \frac{1}{\eta} \delta_{ij} - \frac{1}{\eta^2} \overline{z}_iz_j +\frac{\lambda }{\sigma^2} \delta_{i2}\delta_{j2}, \\
 g_{i\overline{j},k} & = & - \frac{1}{\eta^2} ( \delta_{ij}\overline{z}_k +\delta_{kj}\overline{z}_i)+\frac{2}{\eta^3}\overline{z}_i\overline{z}_kz_j -\frac{2\lambda}{\sigma^3}\overline{z}_2\delta_{i2}\delta_{j2}\delta_{k2}.
 \end{eqnarray*}
At $p=(0,0)$, we have $g_{1\overline{1}}=1$, $g_{1\overline{2}}=0$, $g_{2\overline{2}}=1+\lambda$, and $g_{i\overline{j},k}=0$. So the curvature components at $p$ are given by
$$ R_{i\overline{j}k\overline{\ell} } = - g_{ i\overline{j},k\overline{\ell} } = \delta_{ij} \delta_{k\ell } +  \delta_{i\ell } \delta_{jk} + 2\lambda \delta_{i2} \delta_{j2} \delta_{k2} \delta_{\ell 2}. $$
So for any tangent direction $X$ at $p$, the holomorphic sectional curvature
$$ R_{X\overline{X}X\overline{X}} = 2 (|x_1|^2+|x_2|^2)^2 + 2\lambda |x_2|^4 ,$$
which is always positive. For the Ricci curvature, one has $R_{1\overline{2}}=0$, and
\begin{eqnarray*}
R_{1\overline{1}} & = & R_{1\overline{1}1\overline{1}} + \frac{1}{1+\lambda } R_{1\overline{1}2\overline{2}} \ = \ 2 + \frac{1}{1+\lambda }, \\
R_{2\overline{2}} & = & R_{1\overline{1}2\overline{2}} + \frac{1}{1+\lambda } R_{2\overline{2}2\overline{2}} \ = \ 1 + \frac{1}{1+\lambda } (2+2\lambda ) \ = \ 3.
\end{eqnarray*}
So the Ricci curvature at $p$ is also always positive. This completes the proof of the proposition.


\begin{thebibliography}{A}
		
\bibitem{BishopGoldberg} R.L. Bishop and S.I. Goldberg,  \textit{On the second cohomology group of a K\"ahler manifold of positive curvature,} Proc. Amer. Math. Soc. {\bf 16} (1965), 119--122.


 \bibitem{BZ} Y.-D. Burago and V. A. Zalgaller, \textit{ Geometric inequalities.} Translated from the Russian by A. B. Sosinskii Grundlehren der Mathematischen Wissenschaften [Fundamental Principles of Mathematical Sciences], 285. Springer Series in Soviet Mathematics. Springer-Verlag, Berlin, 1988. xiv+331 pp.


\bibitem{Cao-Ni} H.-D. Cao and L. Ni, \textit{
Matrix Li-Yau-Hamilton estimates for the heat equation on K\"ahler manifolds.}
Math. Ann. \textbf{331} (2005), no. 4, 795--807.


\bibitem{Chau-Tam} A. Chau and L.-F.  Tam, \textit{ K\"ahler C-spaces and quadratic bisectional curvature.} J. Differential Geom. \textbf{94} (2013), no. 3, 409--468.

\bibitem{GHL}
 S. Gallot, D. Hulin and J. Lafontaine, \textit{Riemannian geometry.} Third edition. Universitext. Springer-Verlag, Berlin, 2004. xvi+322 pp.

 \bibitem{Gold} S.-I. Goldberg, \textit{Curvature and Homology.} Revised reprint of the 1970 edition. Dover Publications, Inc., Mineola, NY, 1998. xviii+395 pp.

 \bibitem{GuZhang}
H. Gu and Z. Zhang, {\em An extension of Mok's theorem on the generalized Frankel
conjecture,} Sci. China Math. {\bf 53} (2010), 1--12.

\bibitem{HeierWong}
G. Heier and B. Wong, {\em On projective K\"ahler manifolds of partially positive curvature and rational connectedness,} arXiv:1509.02149.





 \bibitem{HK} E.    Heintze and H. Karcher, \textit{
A general comparison theorem with applications to volume estimates for submanifolds.}
Ann. Sci. \'Ecole Norm. Sup. (4) \textbf{11} (1978), no. 4, 451--470.

\bibitem{Hitchin}
N. Hitchin, {\em On the curvature of rational surfaces,} In Differential Geometry ({\em Proc. Sympos. Pure Math., Vol XXVII, Part 2, Stanford University, Stanford, Calif., 1973}), pages 65-80. Amer. Math. Soc., Providence, RI, 1975.

\bibitem{Hua} L.-K Hua, \textit{ On the theory of automorphic functions of matrix variables I-geometric basis.} Amer. J. Math. \textbf{66} (1944), 470--488.
		
		
\bibitem{HT} S.C. Huang and L.-F. Tam, \textit{$U(n)$-invariant K\"ahler metrics with nonegative quadratic bisectional curvature}, Asian J. Math. \textbf{19}(2015), no. 1,1--16.

\bibitem{Ko} S. Kobayashi, \textit{ On compact K\"ahler manifolds with positive Ricci tensor.} Ann. of Math. \textbf{74} (1961), 570--574.


 \bibitem{Ko2} S. Kobayashi,    \textit{Differential geometry of complex vector bundles.} Publications of the Mathematical Society of Japan, 15. Kano Memorial Lectures, 5. Princeton University Press, Princeton, NJ; Princeton University Press, Princeton, NJ, 1987. xii+305 pp.

 \bibitem{LW} P. Li and J.-P.  Wang,  \textit{Comparison theorem for K\"ahler manifolds and positivity of spectrum.} J. Differential Geom. \textbf{69} (2005), no. 1, 43--74.

\bibitem{LiWuZheng} Q. Li, D. Wu, and F.-Y. Zheng, \textit{An example of compact K\"ahler manifold with nonnegative quadratic bisectional curvature,} Proc. Amer. Math. Soc. {\bf 141} (2013), no.6, 2117--2126.


\bibitem{Liu}G. Liu, \textit{Three-circle theorem and dimension estimate for holomorphic functions on K\"ahler manifolds,} Duke Math. J. \textbf{165} (2016), no. 15, 2899--2919.

 \bibitem{Mi-Pal} V. Miquel and V.  Palmer, \textit{ Mean curvature comparison for tubular hypersurfaces in K\"ahler manifolds and some applications.} Compositio Math. \textbf{86} (1993), no. 3, 317--335.

 \bibitem{MK}  J.   Morrow and K.  Kodaira, \textit{ Complex manifolds.} Holt. Rinehart and Winston, New York-Montreal-London, 1971.


\bibitem{Ni-JDG} L. Ni,  \textit{ Vanishing theorems on complete K\"ahler manifolds and their applications.} J. Differential Geom. \textbf{50} (1998), no. 1, 89--122.		

\bibitem{Ni-Niu2}
 L. Ni and Y.-Y. Niu, \textit{ Gap theorem on K\"ahler manifold with nonnegative orthogonal bisectional curvature,} preprint, arXiv:1708.03534.

		

		
\bibitem{NT2}		L. Ni and L.-F. Tam, \textit{Poincar\'e-Lelong equation via the Hodge-Laplace heat equation.}  Compositio Math. \textbf{149} (2013), 1856--1870.
		
		



		

\bibitem{Niu} Y.-Y. Niu, \textit{A note on nonnegative quadratic orthogonal bisectional curvature}, Proc. Amer. Math. Soc. \textbf{142} (2014), no. 11, 1856--1870.
		


\bibitem{Ni-Wolfson} L. Ni and J.  Wolfson, \textit{
The Lefschetz theorem for CR submanifolds and the nonexistence of real analytic Levi flat submanifolds.}
Comm. Anal. Geom. \textbf{11} (2003), no. 3, 553--564.

\bibitem{Ni-Zheng2} L. Ni and F.-Y Zheng, \textit{ Positivity and Kodaira embedding theorem.} ArXiv preprint:1804.096096.


 \bibitem{petersen} P. Petersen, \textit{ Riemannian geometry.} Third edition. Graduate Texts in Mathematics, 171. Springer, Cham, 2016. xviii+499 pp.

 \bibitem{Sakai} T.     Sakai, \textit{ Riemannian geometry. } Translated from the 1992 Japanese original by the author. Translations of Mathematical Monographs, 149. American Mathematical Society, Providence, RI, 1996. xiv+358 pp.

\bibitem{Schoen-Wolfson} R. Schoen and J.  Wolfson, \textit{ Theorems of Barth-Lefschetz type and Morse theory on the space of paths.} Math. Z. \textbf{229} (1998), no. 1, 77--89.

\bibitem{Tam} L.-F. Tam, \textit{ A K\"ahler curvature operator has positive holomorphic sectional curvature, positive orthogonal bisectional curvature, but some negative bisectional curvature}. Private communication.

 \bibitem{TamYu} L.-F. Tam and C. Yu \textit{ Some comparison theorems for K\"ahler manifolds.} Manuscripta Math. \textbf{137} (2012), no. 3--4, 483--495.


\bibitem{Tian}  G.    Tian, \textit{Canonical metrics in K\"ahler geometry.} Notes taken by Meike Akveld. Lectures in Mathematics ETH Z\"urich. Birkh\"auser Verlag, Basel, 2000. vi+101 pp.

 \bibitem{Tsu}  Y.   Tsukamoto,\textit{ On K\"ahlerian manifolds with positive holomorphic sectional curvature. } Proc. Japan Acad. \textbf{33} (1957), 333--335.
	
\bibitem{Wilking} B. Wilking,  \textit{A Lie algebraic approach to Ricci flow invariant curvature condition and Harnack inequalities,}  J. reine angew. Math. (Crelle), \textbf{679} (2013), 223--247.
    		

\bibitem{wu} H. Wu \textit{ The Bochner Techniques in Differential Geometry.} Classical Topics in Mathematics,  \textbf{6}, High Educational Press, Beijing 2017.

\bibitem{WuYauZheng} D. Wu, S-T Yau, and F.-Y. Zheng, \textit{A degenerate Monge-Amp\`ere equation and the boundary classes of K\"ahler cones,} Math. Res. Lett. {\bf 16} (2009), no.2, 365--374.


\bibitem{WZ} H. Wu and F.-Y. Zheng, \textit{ Examples of positively curved complete K\"ahler  manifolds,} Geometry
		and analysis. No. \textbf{1}, 517--542, Adv. Lect. Math., 17, Int. Press, Somerville, MA, 2011.


\bibitem{XYang}
X. Yang, {\em RC-positivity, rational connectedness, and Yau's conjecture,} arXiv:1708.06713.



		
\bibitem{YZ} B. Yang and F.-Y. Zheng, \textit{$U(n)$-invariant K\"ahler-Ricci flow with nonnegative curvature}, Comm. Anal. Geom., \textbf{21} (2013), no.  2, 251--294.

\bibitem{Yau} S.-T. Yau, \textit{ On the Ricci curvature of a compact K\"ahler manifold and the complex Monge-Amp\`ere equation. I.} Comm. Pure Appl. Math. \textbf{31} (1978), no. 3, 339--411.

 \bibitem{YauP} S.-T. Yau, \textit{ Problem section. } Seminar on Differential Geometry, pp. 669--706, Ann. of Math. Stud., 102, Princeton Univ. Press, Princeton, N.J., 1982.

\end{thebibliography}
\end{document}